\documentclass[11pt,reqno]{amsart}
\usepackage[margin=3cm]{geometry}
\usepackage{amsmath,amssymb,amsthm,amsfonts}
\usepackage{mathtools}
\mathtoolsset{showonlyrefs}
\usepackage{verbatim,epsfig,graphics,hyperref}
\usepackage{marginnote}
\usepackage{color}
\usepackage{mathrsfs}
\usepackage{tikz}
\usepackage{esint}
\usepackage{bigints}
\usepackage{enumitem}
\usepackage{hyperref}
\hypersetup{hidelinks}
\usepackage{dsfont}
\usepackage{float}
\usepackage{enumitem}
\setlist[itemize]{leftmargin=26pt}
\allowdisplaybreaks
\newtheorem{theorem}{Theorem}[section]
\newtheorem{proposition}[theorem]{Proposition}
\newtheorem{corollary}[theorem]{Corollary}
\newtheorem{lemma}[theorem]{Lemma}

\theoremstyle{definition}
\newtheorem{definition}[theorem]{Definition}
\newtheorem{example}[theorem]{Example}
\newtheorem{remark}[theorem]{Remark}
\newcommand{\cC}{\mathscr{C}}
\DeclareMathOperator{\dist}{dist}
\DeclareMathOperator{\var}{Var}
 \newcommand{\cR}{\mathcal{R}}

\newcommand{\cN}{\mathcal{N}}

\newcommand{\half}{{\frac{1}{2}}}
\newcommand{\pa}{\partial}

\DeclareMathOperator{\arsinh}{arsinh}
\DeclareMathOperator{\inj}{inj}
\newcommand{\bc}{\boldsymbol{c}}
\newcommand{\bC}{\boldsymbol{C}}

\newcommand{\eps}{\epsilon}
\newcommand{\injM}{{\inj^\perp_{\scriptscriptstyle \partial M}}}
\newcommand{\RR}{\mathbb{R}}

\author[Chakradhar]{Tirumala Chakradhar}
\address{University of Bristol, School of Mathematics, Fry Building, Woodland Road, Bristol BS8~1UG, U.K.}
\email{tirumala.chakradhar@bristol.ac.uk}

\author[Colbois]{Bruno Colbois}
\address{Universit\'e de Neuch\^atel\\
Institut de Math\'ematiques\\
Rue Emile-Argand 11\\
CH-2000 Neuch\^atel\\
Switzerland
}
\email{bruno.colbois@unine.ch}

\author[Hassannezhad]{Asma Hassannezhad}
\address{University of Bristol,
School of Mathematics,
Fry Building,
Woodland Road,
Bristol, 
BS8~1UG, U.K.}

\email{asma.hassannezhad@bristol.ac.uk}
\title{Lower bounds for the low Steklov eigenvalues}


\begin{document}

\begin{abstract}
For a compact, connected, orientable Riemannian manifold with $b$ boundary components, we obtain geometric lower bounds for the low Steklov eigenvalues, namely $\sigma_k$, $1\le k\le b-1$. Our results complement earlier results, which apply only to $\sigma_k$ with $k\ge b$ and depend on the geometry near the boundary, by showing how the interior geometry influences the low eigenvalues. Our result also yields lower bounds for the low Steklov eigenvalues in the setting of pinched negatively curved manifolds, thus recovering similar results in that context through an alternative proof. The proof of the main result is based on the trace inequality relating the Steklov eigenvalue to the Neumann eigenvalues of the connected subdomains of the manifold containing a boundary collar. The geometric coefficient appearing in this inequality is given by an explicit formula in terms of a quantity that can be interpreted as the electrical resistance of the boundary collar.
\end{abstract}


\maketitle
\setcounter{tocdepth}{1}
\tableofcontents
\section{Introduction} \label{sec: intro}
Let $(M,g)$ be a compact, connected Riemannian manifold of dimension $n$ with smooth boundary $\partial M \neq \varnothing$.  
We consider the Steklov eigenvalue problem on $M$:
\[
\begin{cases}
\Delta u = 0 & \text{in } M,\\[2mm]
\partial_\nu u = \sigma\, u & \text{on } \partial M,
\end{cases}
\]
where $\nu$ denotes the outward unit normal vector field along $\partial M$, and
$\partial_\nu u$ is the corresponding normal derivative.
The spectrum consists of a discrete sequence of real eigenvalues
\[
0=\sigma_0 < \sigma_1 \le \sigma_2 \le \cdots \nearrow +\infty,
\]
each with finite multiplicity.

Let $b$ be the number of connected components of $\partial M$.
We call $\{\sigma_k\}_{k=1}^{b-1}$ the \emph{low Steklov eigenvalues}. These eigenvalues are
especially sensitive to the interaction between the boundary components through the interior
geometry of $M$, and our goal is to obtain geometric lower bounds for them.
For the other eigenvalues $\{\sigma_k\}_{k\ge b}$, a good estimate is given by neighbourhoods of the boundary components through the Dirichlet-Neumann bracketing, see $(2.10)$ and Remark 2.15 in~\cite{CGGS24}. In the same spirit, it was also established in~\cite{CGH20} (see also \cite{Xi,PS19,GKLP}) that, for $k\ge b$, one has a lower bound for the Steklov eigenvalues $\sigma_k$ in terms of the boundary Laplace eigenvalues and the geometry of a small neighbourhood of $\partial M$.  On the other hand, lower bounds for the low-index Steklov eigenvalues are expected to depend on the interior geometry, as demonstrated by the constructions in \cite{CESG19}.\medskip 

One way to obtain lower bounds for $\sigma_1$ is to consider Cheeger-type inequalities. This was first studied by Escobar \cite{Esc97,Esc99} and subsequently by Jammes \cite{Jam15}. An improved form of Jammes' inequality was obtained in  \cite{Per25}. Extensions of Jammes' inequality to higher-order eigenvalues were proved in \cite{HM20}. See also \cite{HS25} for Cheeger--Buser type bounds for $\sigma_1$ in terms of an isocapacitary constant. For a star-shaped Riemannian manifold, a lower bound for $\sigma_1$ was obtained in \cite[Theorem 1.3]{HS20} in terms of certain geometric quantities and the first nonzero Neumann eigenvalue of the domain, extending the results of Kuttler and Sigillito \cite{KS68} for star-shaped Euclidean domains. \medskip

Under suitable sign assumptions on the curvature, several works establish geometric lower bounds that are more explicit than those given by Cheeger-type inequalities. For pinched negatively curved manifolds, lower bounds for the first few Steklov eigenvalues, in terms of the length of a separating multi-geodesic and the volumes of the boundary components, were obtained in \cite{Per25,BBHM,HMP25}. In the nonnegative curvature setting, \cite{XX24} also provides related results in connection with the Escobar conjecture.

Our main results give geometric lower bounds for the low Steklov eigenvalues of a general Riemannian manifold without any a priori curvature assumption. On the one hand, it extends the perspective of \cite{CGH20} by incorporating the influence of the interior geometry; on the other hand, it recovers the similar bounds known in the setting of pinched negatively curved manifolds \cite{HMP25,BBHM}.\smallskip

\noindent{We now state our first result.}
\begin{theorem}\label{thm: intro main thm}Let $(M,g)$ be a compact manifold with smooth boundary. Then for any $L$ less than the normal injectivity radius $\injM$ of $\partial M$, we have
    \[
\sigma_k(M)\ge \frac{1}{4k\,\cR_L}\min\left\{1,\,\sup_{\Omega\in {\mathfrak{U}_L}}\mu_k(\Omega){L^2}\right\}, \qquad k\ge1,
\]
where ${\mathfrak{U}_L}$ is the set of connected subdomains of $M$ containing the boundary collar of width $L$, $\mu_k(\Omega)$ denotes the $k^\text{th}$ positive Neumann eigenvalue of $\Omega$, and $\cR_L$ is the \textit{maximum resistance} in the boundary collar of width $L$, see \eqref{eq: AL BL} for the definition and its connection to the electrical resistance of a one-dimensional conductor.
\end{theorem}
The maximum resistance $\cR_L$ is a geometric constant which can be estimated in terms of curvature bounds. We refer to Proposition \ref{prop: bounds-in-terms-of-curvature} and Appendix \ref{sec: volume density estimate} for more details.\medskip

The Neumann eigenvalue of $\Omega$ in our bound captures the interior dependence. {Estimating  $\mu_k(\Omega)$ might be difficult in general; see \cite[Theorem B]{LY80} for a geometric lower bound for~$\mu_1(\Omega)$; another possible approach is to use Cheeger's inequality \cite{cheeger,buser}, which gives a lower bound in terms of the Cheeger constant (see Remark \ref{ref:cheeger} for the latter). Since the supremum is taken over all $\Omega \in \mathfrak{U}_L$, one may choose the most convenient $\Omega$, for instance $M$ itself. However, this choice may not be optimal. We want to avoid  }
cases where $\mu_k(M)$ can be small but the Steklov eigenvalue is not, such as when we have small necks in the interior that do not disconnect the boundary components, e.g., \cite[Figure~1]{Per25}; see also {Remark \ref{rem:optimality of muk} where we show the optimality of the rate for  $\sup_{\Omega\in \mathfrak{U}_L}\mu_k(\Omega)$ in the lower bound.} \medskip

A natural question is whether the minimum with 1 is needed in Theorem \ref{thm: intro main thm}, or whether the bound can be written only in terms of the second quantity. We construct a family of warped product manifolds {{$M_j$} such that
\[\sigma_1(M_j)\cR_{L_j}\le C\qquad \text{and}\quad \mu_1(M_j)L_j^2\to \infty.\]}
Here, $C$ is a constant independent of $j$. 
In other words, one cannot simply remove the minimum without imposing some conditions. 
See Theorem \ref{thm: manifolds with big mu1} and Figure \ref{fig: manifolds with big mu1} for more details.

We show that, after changing the range of $L$, this obstruction disappears, and one can obtain a lower bound without taking the minimum. Below, we state a special case---where the boundary collar has a warped product structure---of a more general result, Theorem~\ref{thm: general result II}.

Assume that in the boundary collar of width $L_0$, $M$ is isometric to the warped product manifold, {for simplicity of presentation, we assume that the warping functions are the same, i.e. the collar boundary is isometric to} $[0,L_0]\times_{h} \partial M$, where $h\in C^{\infty}([0,L_0])$ and $h(0)=1$. {In this setting, the maximum resistance on the boundary collar of width $L$ is equal to 
\[\cR_L=\int_0^Lh^{1-n}(t)\,dt.\]
Note that $L\mapsto \cR_L$ is strictly increasing. }
\begin{theorem}\label{thm: intro main II}
    Under the assumption above, for any $1\le k\le b-1$ and $L\in(0,L_k]$, where $L_k:=\cR^{-1}\left(\frac{\cR_{L_0}}{4k}\right)$, we have
\[
\sigma_k(M) \ge \frac{1}{4k\cR_L}\sup_{\Omega\in{\mathfrak{U}_L}}\mu_k(\Omega)L^2,
\]
where the supremum is taken over ${\mathfrak{U}_L}$, the set of connected subdomains of $M$ containing the boundary collar of width $L$.
\end{theorem}

In the simplest case, when the boundary collar is isometric to
$[0,L_0)\times \partial M$ with the product metric, there is no need
to restrict the range of $L$. As a consequence of
Theorem~\ref{thm: intro main thm}, for any $L\in (0,L_0)$, we have (See Proposition \ref{cor:k-th stek lower bd for warped bdry})
\begin{equation}
\sigma_k(M) \ge \frac{C}{k} \,\sup_{\Omega\in{\mathfrak{U}_L}}\mu_k(\Omega)L, \qquad 1\le k\le b-1,
\end{equation}
 where $C$ is a universal constant; one may take $C=\frac{1}{\pi^2}$.\smallskip

\begin{remark}
   The dependency of the lower bound on $L$ is essential. Indeed, let $M=[-L,L]\times N$, where $N$ is a closed, connected manifold.  By separation of variables, we know that for $L$ small enough  $$
\sigma_1(M) =\sqrt{\lambda_1(N)}\tanh\left(\sqrt{\lambda_1(N)}{L}\right)\to 0,$$
while $\mu_1(M)=\lambda_1(N)$. Here, $\lambda_1(N)$ denotes the first positive Laplacian eigenvalue of $N$.
\end{remark}

The lower bound in Theorem \ref{thm: intro main thm}  is of particular interest for the low Steklov eigenvalues, namely for $1\le k\le b-1$, where no general lower bounds were previously known. For $k\ge b$, the following lower bound was obtained in \cite{CGH20}; see also \cite{PS17,Xi,GKLP}:
\begin{equation}\label{eq: CGH bound}
\sigma_k(M)\ge \sqrt{\lambda_k(\partial M)+{C(M)}^2}-{C(M)},\qquad k\ge1,
\end{equation}
where $\lambda_k(\partial M)$ is the $k$th Laplace eigenvalue of the boundary, and {$C(M)>0$ is a constant depending only on the sectional curvature bounds in the boundary collar, the principal curvature bounds of the boundary, and the normal injectivity radius}. Although inequality \eqref{eq: CGH bound} is stated for all $k\ge1$, it is non-trivial only for $k\ge b$, since the first $b$ Laplace eigenvalues of $\partial M$ are zero.  In fact lower bound \eqref{eq: CGH bound} is even more interesting for large $k$ as the lower bound is asymptotically sharp.  Thus, Theorem \ref{thm: intro main thm} completes the picture by providing lower bounds for the remaining low Steklov eigenvalues $\sigma_k$, $1\le k\le b-1$.
\begin{remark}\label{ref:cheeger}
To our knowledge, in the absence of additional geometric assumptions, such as curvature
bounds or star-shapedness, the only known lower bounds for $\sigma_1(M)$ when the
number of boundary components $b$ is at least two are Cheeger-type lower bounds,
obtained by Escobar~\cite{Esc97}, and by Jammes~\cite{Jam15}, with a recent
improvement in~\cite{Per25}. Under additional assumptions, other lower bounds are
available, but these require some geometric constraint on the manifold or on the
boundary. Escobar's result depends on an auxiliary Robin problem and on an
isoperimetric ratio. Jammes gave a simplified version involving only two
isoperimetric ratios, which we recall below in order to compare it with our result.
The inner and exterior boundaries of a domain $D \subset M$ are defined by
\[
\partial_I D := \partial D \cap \operatorname{int} M,
\qquad
\partial_E D := \overline{D} \cap \Sigma .
\]
The improved Cheeger and Jammes constants of $M$ are defined by
\[
\tilde h_M = \inf_D \frac{|\partial_I D|}{|D|}
\qquad \text{and} \qquad
\tilde h^J_M = \inf_D \frac{|\partial_I D|}{|\partial_E D|},
\]
where both infima are taken over all domains $D \subset M$ such that
$|D| \leq \frac{|M|}{2}$, and such that both $D$ and its complement $D^c$ contain part of the boundary. Note that the original definition of Cheeger constant $h_M$ and Jammes constant $h_M^J$ is the same but does not require that $D$ and $D^c$ contain part of the boundary. In~\cite{Per25} (see also \cite{Jam15}), it is shown that
\[
\sigma_1(M) \geq \frac{1}{4} \tilde h_M \tilde h^J_M .
\]
\begin{figure}
    \centering
  \noindent\makebox[\linewidth][c]{%
  \def\svgwidth{.55\linewidth}%
  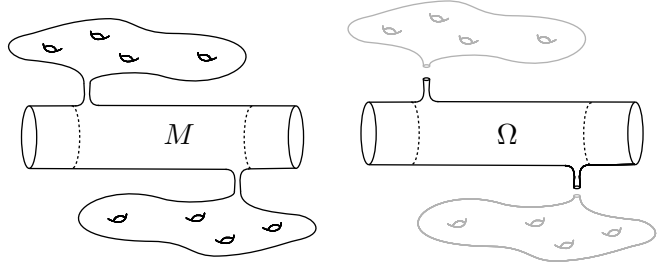%
}
    \caption{This picture illustrates that $\tilde h_M$ can tend to zero when the volumes of the two regions connected by a neck tend to $\infty$, while $ h_\Omega$ remains unchanged, where $\Omega$ is obtained from $M$ by removing those two regions. Note that, to obtain an upper bound for $\tilde h_M$, it is enough to cut through the middle of the cylinder; then each resulting part contains a boundary component.}
    \label{fig:tildehM}
\end{figure}
On the other hand, if in Theorem~\ref{thm: general result II} we replace
$\mu_1(\Omega)$ by the Cheeger lower bound $\mu_1(\Omega) \geq \frac{1}{4}h_\Omega^2$, 
then we obtain
\[
\sigma_1(M)
\geq
\frac{1}{16\cR_L}
\sup_{\Omega \in \mathfrak{U}_L} h_\Omega^2 L^2,
\]
where $L$ may be any positive number in a reduced range that is defined explicitly in terms of the resistance. We refer to Theorem~\ref{thm: general result II} for the precise definition.

The above inequality is valid for $L$ in a range depending only on the geometry of
the boundary collar; see Theorem~\ref{thm: general result II} for the precise range
of $L$. Thus, in this estimate, the global quantity $\tilde h^J_M$ is replaced by a geometric quantity $\frac{L^2}{R_L}$ depending only on the geometry of the boundary collar. Moreover, the isoperimetric ratio $\tilde h_M$ is
replaced by $\sup_{\Omega \in \mathfrak{U}_L} h_\Omega$, where
$\mathfrak{U}_L$ consists of connected domains containing the boundary collar. This quantity can be much larger than $\tilde h_M$, since the supremum is taken over a
restricted class of domains rather than over all of $M$. See Figure \ref{fig:tildehM} for an illustration.  We can make a similar comparison between our lower bound and the higher-order
Cheeger inequalities obtained in~\cite{HM20} for $2 \leq k \leq b-1$.
\end{remark}

The proofs of Theorem \ref{thm: intro main thm} and the lower bound \eqref{eq: CGH bound} are also related in spirit. In~\cite{PS19,Xi,CGH20}, the authors use the eigenfunctions corresponding to $\{\sigma_i(M)\}_{i=1}^k$ as test functions, but for $\lambda_k(\partial M)$, and apply the Pohozaev identity to relate the Rayleigh quotients associated with the two operators on the boundary. As a result, they get a non-trivial lower bound for $\sigma_k(M)$ only when $k\ge b$. In our case, the comparison is instead made with the Neumann eigenvalues of a subdomain of the manifold, and we replace the role of the Pohozaev identity  by using a trace inequality (Lemma \ref{lem: trace estimate for general collars}). 

A closely related trace inequality appears in the proof of Theorem 1.3 in \cite{HS20}, where it is used to obtain a lower bound for $\sigma_1$ in terms of $\mu_1$ on star-shaped domains. We refer to Theorem \ref{thm: main result second version}, where we obtained a different lower bound using a trace inequality ({see} Lemma~\ref{lem:trace type II}) related to that used in \cite{HS20}.  \medskip

{Our results not only extend the method of \cite{CGH20} to low Steklov eigenvalues, but also provide a unified framework for deriving lower bounds, which can be applied in the important setting of pinched negatively curved manifolds.} We start with the case of  hyperbolic surfaces. Let \(M\) be a compact hyperbolic surface with geodesic boundary. We denote by \(\gamma\) the genus of \(M\), by \(b\) the number of boundary components, and by \(\beta\) the maximum length of the boundary components. By the collar theorem, the boundary collar has a warped product structure. Applying Theorem \ref{thm: intro main II} together with an extension of Dodziuk--Randol's lower bound for \(\mu_1\) to hyperbolic manifolds with geodesic boundary (Lemma \ref{lem: DR mu k}), we obtain the following lower bound.
\[ \sigma_k(M,g)\ge \frac{C(b,\gamma)}{k}\min\left\{\frac{1}{e^{\beta/2}},\frac{\ell_k}{e^\beta}\right\},\qquad 1\le k\le b-1,
\]
where $\ell_k$ is the infimum of the total lengths of all multi-geodesics lying away from $\partial M$ and separating $M$ into $k+1$ connected components, each of which contains a boundary component. If there is no such multi-geodesic, then $\ell_k=\infty$. 

\noindent It almost recovers the result of \cite{HMP25}: \begin{align}\label{eq:HMP}
\sigma_k(M, g)
\ge C(b,\gamma)\min\left\{\frac{1}{(1+\beta)^2e^{\beta}},\frac{\ell_k}{\beta}\right\},\qquad 1\le k\le b-1.
\end{align}

The main difference between the two results is the dependence on $\beta$ on the right-hand side. In \cite{HMP25}, it is shown that this dependence on $\beta$ is optimal when $\ell_k \to 0$. However, our result provides a new proof without needing an adaptation of the thick-thin decomposition  of the method of \cite{DR86} to the Steklov problem. The lower bound extends to pinched negatively curved surfaces;  see Theorem \ref{prop: HMP bound}.

As proved in \cite[Theorem 1.3]{BBHM}, pinched negatively curved manifolds of
dimension $n \geq 3$ have a positive Steklov spectral gap when the volume is
bounded. It can be considered a counterpart of Schoens' estimate \cite{Schoen82} for the Laplacian. More precisely, the authors proved that
\begin{equation}\label{eq: BBHM}
\sigma_1(M) \geq
\frac{C(n,\delta)}
{b\,|M|^2\,\beta^{\frac{2n}{\delta(n-2)}}},
\end{equation}
where $|M|$ is the Riemannian volume of $M$.

As an application of Theorem \ref{thm: intro main II}  together with 
Schoen's spectral gap estimate for the Laplacian~\cite{Schoen82}, we {obtain a new lower bound} similar to \eqref{eq: BBHM} which improves the rate
of dependence on $\beta$ as $\beta\to\infty$. 
\begin{theorem}
    Let $(M^n,g)$, $n \ge 3$, be a compact connected Riemannian manifold whose sectional curvatures lie in $[-1,-\delta^2]$ for some $\delta \in (0,1]$. Assume that $\partial M$ has $b$ totally geodesic connected components $\Sigma_j$, $1 \le j \le b$, and set $\beta := \max_{1 \le j \le b} |\Sigma_j|$. Then
    \[\sigma_1(M)\ge\frac{C(n,\delta)}{|M|^2\,\beta^{\frac{2}{\delta(n-2)}}}.\]
\end{theorem}

When $\beta$ is bounded and $|M|\to\infty$, a sharper estimate in this
particular regime is available in the refined version of \eqref{eq: BBHM};
see \cite[Theorem~4.4]{BBHM}. There, the power of $|M|$ in the denominator is
$1$, and this power is optimal \cite[Example~5.1]{BBHM}. Our approach is aimed
at a different advantage: it avoids the heavy technical details of the proof
of \cite[Theorem 4.4]{BBHM} {related to adapting the thick-thin decomposition, and the Dodziuk-Randol argument to the Steklov setting,} while still recovering the main estimate on the spectral gap through a substantially
{shorter proof}. \medskip

The lower bounds \eqref{eq: CGH bound} and \eqref{eq:HMP}, obtained in
\cite{CGH20} and \cite{HMP25}, respectively, are accompanied by corresponding upper bounds. In \cite{CGH20}, the Steklov eigenvalues are compared
with the square roots of the Laplace eigenvalues on the boundary; namely,
\[
\left|\sigma_k(M)-\sqrt{\lambda_k(\partial M)}\right|\leq C_1(M),\qquad k\ge 1
\]
where $C_1(M)$ up to  constant depending only on the dimension is the same constant as in~\eqref{eq: CGH bound}. Here, $\lambda_k(\partial M)$ denotes the $(k+1)$-st eigenvalue of the Laplacian on $\partial M$, with the indexing starting at $0$. Similarly, in the setting of hyperbolic surfaces, it is shown in
\cite{HMP25,Per25} that, when $\ell_k$ is sufficiently small, the ratio
$\sigma_k/\ell_k$ is bounded above and below by constants depending only on the
topology and the boundary length. See also \cite{HS25}, where the
isocapacitary constant appears in both the lower and upper bounds.
It is therefore natural to ask whether one can also obtain an upper bound of
the form
$${\sigma_1(M)\le\frac{C_2(M)}{\cR_L}}\sup_{\Omega\in{\mathfrak{U}_L}}\mu_1(\Omega)L^2,$$
for some geometric constant $C_2(M)$ depending only on the geometry of the
boundary collar.  {We answer this question negatively in Theorem~\ref{thm:counterexample upper bound}. More precisely,} we show that 
    there exists a sequence of compact Riemannian manifolds $(M_j^n,g_j)$ whose boundary collars are isometric to
$[0,1]\times \partial M_j$, and whose normal injectivity radii are uniformly
bounded below, such that
\[
    \frac{\sigma_1(M_j,g_j)\cR_L}
    {\sup_{\Omega\in\mathfrak{U}_L}\mu_1(\Omega)L^2}
    \longrightarrow \infty
    \qquad \text{as } j\to\infty, 
\]
{for any $L$ smaller than the normal injectivity radius. Our proof is based on obtaining a lower bound for the above quotient on a family of product manifolds. We remark that, in general, controlling $\sup_{\Omega\in\mathfrak{U}_L}\mu_1(\Omega)$ from above is difficult. As we shall see, even in the product setting the argument is non-trivial.}
\medskip

The paper is organised as follows. The main results are presented in two parts. In~Section~\ref{sec: main results}, we state the results in a general setting. We then discuss the {main applications of our  results in the general setting} to manifolds whose boundary collar has a warped product structure and to manifolds with pinched negative curvature in Section \ref{sec: 3}. Section \ref{sec: proofs in hyperbolic setting} contains the proofs in these settings. In Section \ref{sec: proof of examples}, we show that the minimum appearing in Theorem \ref{thm: intro main thm} cannot be removed unless the range of $L$ is reduced. We also show that the same quantity appearing in the lower bound cannot, in general, be used to bound the Steklov eigenvalues from above. Both theorems are proved by constructing families of manifolds with warped product structure near the boundary. The results in general setting are proved in Section \ref{sec: proof of main theorem}. Finally, in the appendix, we provide geometric estimates, in terms of curvature bounds, for the quantities appearing in the lower bound in the general setting, using the Jacobi comparison theorem.

\section{Statements of the Main Results{---General Setting}}\label{sec: main results}
{We begin by introducing some notation and recalling a few definitions related to the geometry of the boundary collar.}

Let $(M^n,g)$ be a compact, connected, orientable smooth Riemannian manifold with its boundary comprising $b\ge 1$ connected components $\{\Sigma_j\}_{j=1}^b$.
Take  $L>0$  be such that the normal exponential map is a diffeomorphism on $$\cN_L(\partial M):=\{(-t\nu_x,x)\in T_x\partial M^{\perp}: 0\le t\le L,~x\in \partial M\}.$$ 
$$\exp^{\perp}:\cN_L(\partial M)\to M, \qquad\exp^{\perp}(-t\nu,x)=\exp_x(-t\nu),$$ where $
\nu$ is the  outward unit normal to the boundary. {The supremum of such values of $L$ for which $\exp^{\perp}$ defines a diffeomorphism is called the \textit{normal injectivity radius} of $\partial M$, and we denote it by $\injM$.}

We write the metric in the Fermi coordinates based on each boundary components $\Sigma_j$  as
$g = dt^2 + g_t^{(j)}$, where $t\in[0,L]$ denotes the distance to $\Sigma_j$
and $g_t^{(j)}$ is the induced metric on the level set $\Sigma_j^t=\{x\in M:\mathrm{dist}(x,\Sigma_j)=t\}$.
Denote by
\begin{equation}\label{def: varphi}
\varphi_j(t,x) := \sqrt{\frac{\det g_t^{(j)}(x)}{\det g_0^{(j)}(x)}}
\end{equation}
the volume density ratio along the half-collar ${\mathscr{C}_j^{L}} = \{x\in M:\dist(x, \Sigma_j)\le L\}$,
so that $dV_g = \varphi_j\,dt\,dA_{\Sigma_j}$ on ${\mathscr{C}_j^{L}} $ in Fermi coordinates. 
For every $x\in \Sigma_j$, we define the \textit{resistance along the ray} $[0,L]\times\{x\}$  as $$\cR_L^{(j)}(x):=\int_0^L \varphi_j(t,x)^{-1}\,dt,$$ 
and the maximum resistance on $\mathcal{C}_L:=\sqcup_{j=1}^b \mathscr{C}_j^{L}$ as 
\begin{equation}\label{eq: AL BL}
\cR_L := \max_{1\le j\le b}\sup_{x\in\Sigma_j}\int_0^L \varphi_j(t,x)^{-1}\,dt.
\end{equation}
The terminology is inspired by the definition of electrical resistance in
physics. For a one-dimensional conductor with variable cross-sectional area $A(t)$ the infinitesimal resistance is proportional to $\frac{dt}{A(t)}$ and
the total resistance is obtained by integrating the reciprocal of the
cross-sectional area. In the present setting, the volume density ratio  $\varphi_j(t,x)$ plays the role of the cross-sectional area along the ray $[0,L]\times\{x\}$. Hence, $\cR_L^{(j)}(x)$ has the natural
interpretation of an effective resistance along that ray and the geometric constant $\cR_L$ as the maximum restistance of the boundary collar.\smallskip

In the context of electrical impedance
tomography, the Dirichlet-to-Neumann map is often called the
voltage-to-current map, since it sends an imposed boundary voltage to the
corresponding induced current flux through the boundary; see, for instance, \cite{Uhl}. This interpretation can
explain why $\cR_L^{(j)}(x)$  and $\cR_L$ are natural quantities in the Steklov setting.

\begin{theorem}\label{thm:general result}
     Under the assumption and notations above, for every $L\in (0,\injM]$  
\[
\sigma_k(M) \ge \frac{1}{4k\,\cR_L}\min\left\{1,\,\sup_{\Omega\in{\mathfrak{U}_L}}\mu_k(\Omega)L^2\right\}, \qquad k\ge1,
\]
where ${\mathfrak{U}_L}$ is the set of connected subdomains of $M$ containing $\mathcal{C}_L$, and $\mu_k(\Omega)$ is the $k^\text{th}$ positive Neumann eigenvalue of $\Omega$.
\end{theorem}

The lower bound in Theorem \ref{thm:general result} (and also Theorem \ref{thm: main result second version}) is of particular interest when $1\le k\le b-1$, since no general lower bounds are known for low Steklov eigenvalues. For $k\ge b$, there are known lower bounds depending only on the geometry near the boundary, as mentioned in the introduction.

We can ask whether we can remove the minimum in the lower bound.  The following theorem shows that this is not the case.
\begin{theorem}\label{thm: manifolds with big mu1}
    Given $L>0$, there exists a family of warped product manifolds $M_j$ with $\inj^{\perp}(\partial M_j)=L$ such that $\sigma_1(M_j)\cR_{L,j}$ is uniformly bounded from above but $\mu_1(M_j)L^2$ tends to infinity. When $M_j$ is of dimension $\ge 3$, we can further assume that its boundary has a fixed volume.
\end{theorem}
Figure \ref{fig: manifolds with big mu1} illustrates the construction of such sequence of $M_j$. The idea is to choose the warping function so that the weight is concentrated near the middle slice. This forces $\mu_1(M_j)$ to be large. On the other hand, the energy of the radial Steklov profile is controlled by {$\frac{1}{\cR_{L,j}}$}.
\begin{figure}[H]
    \centering
    \input{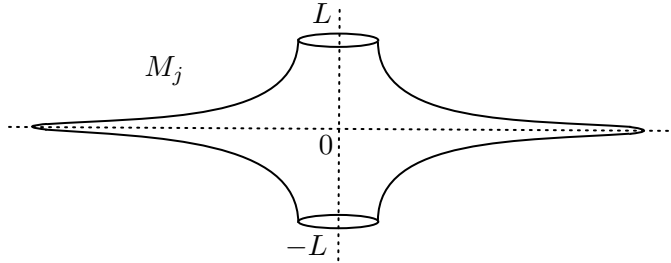}
    \caption{An illustration of $M_j$ in the proof of Proposition \ref{thm: manifolds with big mu1}.}
    \label{fig: manifolds with big mu1}
\end{figure}\bigskip

{Theorem~\ref{thm:general result} holds for any value of $L$ in $(0,\injM]$.
We shall see that restricting the admissible range of $L$ allows one to remove
the minimum on the right-hand side. In particular, for the
family of warped product manifolds constructed in the proof of
Theorem~\ref{thm: manifolds with big mu1}, $\mu_1(M_j)L^2$ and
$\sigma_1(M_j)\cR_{L,j}$ have the same rate in the reduced range of $L$; see
Remark~\ref{rem:Lk-for-M-j}.}\medskip 

To state the result, we need to introduce a
\textit{distortion condition} in the boundary collar {which is a way to  measure how far the boundary collar is from a warped product. }\\
{\noindent We define the average density of the slice \(\Sigma_j^{t}=\{t\}\times\Sigma_j\) by
\[
\overline\varphi_j(t)
:=
\frac{1}{|\Sigma_j|}
\int_{\Sigma_j}\varphi_j(t,x)\,dv_{\Sigma_j}(x),\qquad \text{
or equivalently},\qquad
\overline\varphi_j(t)
=
\frac{|\Sigma_{j}^t|}{|\Sigma_j|}.
\] 
\begin{definition}\label{def: uniform distorsion}
    {For $\tau\ge 1$ and $L\in(0,\injM]$, we say that $M$ satisfies the
\textit{$\tau$-distortion condition} in $\mathcal{C}_L$ if, for every
$1\le j\le b-1$ and every $(t,x)\in [0,L]\times \Sigma_j$,}
\[
\tau^{-1}\overline\varphi_j(t)
\leq
\varphi_j(t,x)
\leq
\tau\overline\varphi_j(t).
\]
\end{definition}
\medskip}

\noindent
Because of compactness of $M$ and positivity of $\varphi_j$, the quantity
\begin{equation}\label{def:tauM}
\tau_L(M)
:=
\max_{1\le j\le b-1}
\sup_{(t,x)\in[0,L]\times\Sigma_j}
\max\left\{
\frac{\varphi_j(t,x)}{\overline\varphi_j(t)},
\frac{\overline\varphi_j(t)}{\varphi_j(t,x)}
\right\}
\end{equation}
is finite. Thus $M$ satisfies the $\tau$-distortion condition in
$\mathcal C_L$ for any $\tau\ge\tau_L(M)$ and $L\in(0,\injM]$. {Note that $\tau_L(M)=1$ when we have a warped product structure on the boundary collar of width $L$.}

\begin{remark}
    Let $M$ satisfy the $\tau$-distortion condition in $\mathcal{C}_L$. Define
\[
\overline \cR_L
:=
\max_{1\le j\le b}
\int_0^L \overline\varphi_j(t)^{-1}\,dt,\qquad \underline \cR_L
:=
\min_{1\le j\le b}
\int_0^L \overline\varphi_j(t)^{-1}\,dt.
\]
Then
\begin{equation}\label{eq:overline AL and BL bounds AL}
    \overline{\cR}_L\le \cR_L
\leq
\tau\,\overline \cR_L.
\end{equation}
Indeed, since $f(x)=x^{-1}$ is convex, by Jensen's inequality we have 
\[
\overline\varphi_j(t)^{-1}
=
\left(
\frac{1}{|\Sigma_j|}
\int_{\Sigma_j}\varphi_j(t,x)\,dv_{\Sigma_j}(x)
\right)^{-1}
\leq
\frac{1}{|\Sigma_j|}
\int_{\Sigma_j}\varphi_j(t,x)^{-1}\,dv_{\Sigma_j}(x).
\]
Integrating in $t$, we obtain
\[
\int_0^L \overline\varphi_j(t)^{-1}\,dt
\leq
\frac{1}{|\Sigma_j|}
\int_{\Sigma_j}
\int_0^L \varphi_j(t,x)^{-1}\,dt
\,dv_{\Sigma_j}(x)
\]
which implies, $
\overline \cR_L\leq \cR_L$. The other inequality immediately follows from Definition \ref{def: uniform distorsion}.\end{remark}

\begin{theorem}\label{thm: general result II}
   {Let $\tau\ge 1$ and let $L_0\in(0,\inj^{\perp}(\partial M)]$. Assume that
$(M,g)$ satisfies the $\tau$-distortion condition in $\mathcal{C}_{L_0}$.
Define $L_k(\tau,L_0)\in(0,L_0]$ by
\[
\overline{\cR}_{L_k(\tau,L_0)}
=
\frac{\underline{\cR}_{L_0}}{4k\tau}.
\]
Then, for any $L\in(0,L_k(\tau,L_0)]$, we have
}
\[
\sigma_k(M) \ge \frac{1}{4\cR_Lk}\sup_{\Omega\in{\mathfrak{U}_L}}\mu_k(\Omega)\, L^2, \qquad 1\le k\le b-1.
\]
\end{theorem}
\begin{remark}\label{rem: dist warp product} 
 When $M$ has a warped product structure in $\mathcal{C}_{L_0}$, i.e. for every $1\le j\le b$,  $\cC_j^{L_0}$ is isometric to the warped product manifold $[0,L_0]\times_{h_j} \Sigma_j$, where $h_j\in C^{\infty}([0,L_0])$ and $h_j(0)=1$. {As mentioned above,} $M$ satisfies $1$-distortion condition {in $\mathcal{C}_L$} for any $L\in(0,L_0]$. However, even
in the warped product setting, the ratio $\frac{L_k(1,L_0)}{L_0}$ can be
arbitrarily small, depending on the warping function. {See Remark \ref{rem:Lk-for-M-j} where we compute $L_k(1,L_0)$ for the family of warped product manifolds used in the proof of Theorem \ref{thm: manifolds with big mu1}}. 
\end{remark}

Using the Jacobi field comparison theorem, we derive estimates for the volume density ratio $\varphi_j$ in terms of the normal injectivity radius, sectional-curvature bounds in $\mathcal{C}_L$, and bounds on the principal curvatures of the boundary; see Appendix~\ref{sec: volume density estimate}. Applying these estimates yields an upper bound for $\cR_L$ and a lower bound for the reduced range of $L$, leading to the following statement. 
\begin{proposition}\label{prop: bounds-in-terms-of-curvature} Let $\mathrm{Sec}_{g}$ denote the sectional curvature of $M$. Assume that on $\mathcal{C}_{\injM}$ 
\begin{itemize}
    \item[(H1)] {$K_{\min}\le \mathrm{Sec}_{g}\le K_{\max}$ for some
$K_{\min},K_{\max}\in \RR$, and}
\item[(H2)] the principal curvatures of $\Sigma_{j}$ with respect to the
outward normal $\nu$ satisfy
{
\[
\kappa_{\min}\le \kappa_{i}(x)\le \kappa_{\max},\;\;1\le i\le n-1,\;x\in\Sigma_{j},
\]}
for some $\kappa_{\min},\kappa_{\max}\in \mathbb{R}$.
\end{itemize}
 Then there exist constants ${ \hat L:= \hat L}(n,{K_{\min},K_{\max},\kappa_{\min},\kappa_{\max}},{\injM})>0$ and \\${C:= C}(n,{K_{\max},\kappa_{\max}},{\injM})>0$ such that, for every $L\in(0,\hat L]$,
\[
\sigma_k(M) \ge \frac{{C}}{4k}\sup_{\Omega\in{\mathfrak{U}_L}}\mu_k(\Omega)L^2, \qquad 1\le k\le b-1;
\]
\end{proposition}
\noindent See Appendix~\ref{sec: volume density estimate} for explicit expressions for {$L_0$ and $C$}.\medskip

The main ingredient in the proofs of Theorems~\ref{thm:general result} and~\ref{thm: general result II} is an explicit geometric estimate for the constant in the trace inequality. We postpone the proof to Section~\ref{sec: proof of main theorem}. Different geometric estimates for this trace constant lead to different lower bounds. Although we do not use the following result elsewhere, it illustrates another lower bound obtained from the same trace inequality, using a different geometric estimate for its constant; see Lemma~\ref{lem:trace type II}. 

\begin{theorem}\label{thm: main result second version} In the notation of Proposition \ref{prop: bounds-in-terms-of-curvature}, Let {$\kappa_{\max}\in\RR_{\ge0}$} be {an} upper bound for the principal curvature of $\partial M$ and {$K_{\max}\in\RR$} an upper bound for the sectional curvature. Then there exists $L_*=L_*(K_{\max},\kappa_{\max},\injM)$ such that for any $L \in (0,L_*]$
\[\sigma_k(M)\ge\frac{1}{4k}\sup_{\Omega\in{\mathfrak{U}_L}}\min\left\{{\sqrt{\mu_1(\Omega)}},\frac{2L}{1+nL\kappa_{\max}}\mu_k(\Omega)\right\},\qquad k\ge 1.\]
One can take $L_*=\min\{L^\star, \injM\}$, where $L^\star$ is defined in Appendix \ref{sec: volume density estimate}.
\end{theorem}

\noindent Note that the minimum is attained by the second term when \[\sup_{\Omega\in{\mathfrak{U}_L}} \mu_1(\Omega) < \left(\frac{1 + nL\kappa_{\max}}{2L}\right)^2.\]

\section{{Warped Boundary Collars and Pinched Negatively Curved Manifolds}}\label{sec: 3}

We now consider consequences of Theorems~\ref{thm:general result} and~\ref{thm: general result II} in two geometric settings: manifolds with a warped product structure near the boundary, and pinched negatively curved manifolds with totally geodesic boundary.\medskip

\paragraph{\textbf{Notation.}~}
In what follows, \(C\) and \(C_j\) denote universal positive constants, while
\(C(\cdot)\) and \(C_j(\cdot)\) denote positive constants depending only on the
quantities in their arguments. The same notation may be used to represent different constants. 

{\subsection{{Warped Product Manifolds} }
We say that $M$ has a warped product structure near the boundary if there exists $L_0>0$ such that, for every $1\leq j\leq b$, the half-collar $
\mathscr{C}_j^{L_0}=\{x\in M:\operatorname{dist}(x,\Sigma_j)\leq L_0\}$
is isometric to the warped product manifold $[0,L_0]\times_{h_j}\Sigma_j$, where $h_j\in C^\infty([0,L_0])$ and $h_j(0)=1$. By Remark~\ref{rem: dist warp product}, it satisfies the $1$-distortion condition.  

For simplicity of presentation, we assume that the warping functions are the same, that is, $h_j=h$ for all $1\leq j\leq b$. Hence, 
\begin{equation}\label{eq:warped-bar-AL}
    \overline{\cR}_L=\underline \cR_L=\cR_L=\int_0^Lh^{1-n}(t)\,dt,\end{equation}
and since $L\mapsto \cR_L$ is striclty increasing, we can write  
\begin{equation}\label{eq:warped-Lk}
L_k(L_0):=L_k(1,L_0)=\cR^{-1}\left(\frac{\cR_{L_0}}{4k}\right).\end{equation}

Hyperbolic manifolds with totally geodesic boundary provide important examples of manifolds with a warped product structure near the boundary, with warping function $h(t)=\cosh t$ and volume density ratio
$\varphi_j(x,t)=\cosh^{n-1} t$. {In the following examples we give a lower bound for  $L_k(\injM)$ when $M$ is a hyperbolic manifold with totally geodesic boundary. }
\begin{example}[Hyperbolic surfaces]\label{ex:hypsurface} Let $M$ be a hyperbolic surface with $b$ geodesic boundary components. Let $\beta$ denote the maximum length of its boundary components.
     By the collar theorem (see e.g. \cite{Bus92}) we can take 
\begin{equation}\label{eq:lower bound for L}
{\injM\ge} L_{ \beta}:=\arsinh\left(\frac{1}{\sinh(\beta/2)}\right)=\ln\coth{\frac{\beta}{4}}\ge {C_1}{e^{-\beta/2}},
\end{equation} 
Let us calculate $L_{k,\beta}:=L_k(L_\beta)$. We have
\[
L_{k,\beta}
=
\operatorname{arcsinh}
\left[
\tan\left(
\frac{1}{4k}
\arctan\left(\frac{1}{\sinh(\beta/2)}\right)
\right)
\right],
\]
and as $\beta\to 0$, we have $
L_{k,\beta}\to
\operatorname{arcsinh}\left(\tan\frac{\pi}{8k}\right).$
As $\beta\to\infty$, we have
$
L_{k,\beta}
\sim
\frac{1}{2k}e^{-\beta/2}.
$ Thus there exists a universal constant $C_2>0$ such that for all $\beta>0$ and all $k\geq 1$
\[
{L_k(\injM)\ge} L_{k,\beta}
\geq
\frac{C_2}{k}e^{-\beta/2}.
\]
\end{example}
\begin{example}[Hyperbolic $n$-manifolds, $n\ge3$]\label{ex:hyperbolic-n-mfld} Let $M$ be a hyperbolic manifold  of dimension $n\ge3$ with $b$ totally geodesic boundary components $\Sigma_j$.
 Set $$\beta := \max_{1 \le j \le b}  |\Sigma_j|,\qquad \text{and}\qquad \rho(r):=\log\coth\frac{r}{2}$$
 Let $V_{n-1}(r)$ denote the volume of a ball of radius $r$ in the hyperbolic space  $\mathbb{H}^{n-1}$. Then  by the collar theorem \cite{Bas94}, we can take $L_0$ in Corollary \ref{cor:k-th stek lower bd for warped bdry} as 
$${\injM\ge} L_\beta:=\frac{1}{2}(V_{n-1}\circ \rho)^{-1}(\beta)\ge C_1(n)\beta^{-\frac{1}{(n-2)}}.$$
For the last inequality, see e.g. \cite{Bas94} and \cite[Lemma 3.5, Remark 3.6]{BBHM}.
We now calculate $L_{k,\beta}:=L_k(L_\beta)=
\cR^{-1}\left(\frac{\cR_{L_\beta}}{4k}\right)$. Since
$\cosh^{1-n}t\leq 1$, we have $\cR_L\leq L$, and hence
\[
{L_k(\injM)\ge} L_{k,\beta}
=
\cR^{-1}\left(\frac{\cR_{L_\beta}}{4k}\right)
\geq
\frac{\cR_{L_\beta}}{4k}.
\]
Moreover,
\[
\cR_{L_\beta}
=
\int_0^{L_\beta}\cosh^{1-n}t\,dt
\geq
\int_0^{\min\{L_\beta,1\}}\cosh^{1-n}t\,dt
\geq
\cosh(1)^{1-n}\min\{L_\beta,1\}.
\]
Therefore,
\[
L_{k,\beta}
\geq
\frac{\cosh(1)^{1-n}}{4k}\min\{L_\beta,1\}\ge \frac{C_2(n)}{k}\beta^{-\frac{1}{n-2}}.
\]
\end{example}
{We now use the above estimates and obtain a lower bound for the low Steklov eigenvalues of hyperbolic  $n$-manifolds, $n\ge2$,  with totally geodesic boundary in terms of $\beta$ and $\mu_k(\Omega)$.
\begin{corollary}\label{cor:hyperbolic} Let $M$ be a hyperbolic $n$-manifold with totally geodesic boundary. Let $\beta$ denote the maximum volume of its boundary components. Then for any $1\le k\le b-1$
\begin{itemize}
    \item[(a)] if $n=2$ we have \begin{equation}\label{eq:lower-bound-hyperbolic-surfaces}
    \sigma_k(M) \ge \frac{C_3}{k^3\,e^{\beta}}\sup_{\Omega\in{\mathfrak{U}_{L_{k,\beta}}}}\mu_k(\Omega),
\end{equation}
    \item[(b)] and if $n\ge3$
    \begin{equation}\label{eq:lower-bound-hyperbolic-n-manifolds}
    \sigma_k(M) \ge \frac{C_3(n)}{k\,{\beta}^{\frac{1}{n-2}}}\sup_{\Omega\in{\mathfrak{U}_{L_{k,\beta}}}}\mu_k(\Omega),
\end{equation}
\end{itemize}
where ${\mathfrak{U}_{L_{k,\beta}}}$ is the set of connected subdomains of $M$ containing the boundary collar $\mathcal{C}_{L_{k,\beta}}$, and $L_{k,\beta}$ is defined in the above examples.  
\end{corollary}
The above corollary is an immediate consequence of Theorem \ref{thm: general result II}  by taking $L=L_{k,\beta}$ and using the bounds we obtain in Examples \ref{ex:hypsurface} and \ref{ex:hyperbolic-n-mfld} together with the following upper bound.
\[
   \cR_{L_{k,\beta}}=\int_0^{L_{k,\beta}}\cosh^{1-n}t\,dt\le\frac{2^{n-1}}{n-1} .
\]}
Estimating $L_k(L)=\cR^{-1}\left(\frac{\cR_L}{4k}\right)$ for a general warping function may not be straightforward. We give another criterion for manifolds with a warped product structure in the boundary collar, under which the minimum in Theorem \ref{thm:general result} can be removed.
\begin{proposition}\label{cor:k-th stek lower bd for warped bdry}  Assume that $M$ has a warped product structure on $\mathcal{C}_{L_0}$ with the same warping function $h\in C^{\infty}([0,L_0])$ with $h(0)=1$ along all boundary components. 
If there exists $C(n)>0$ independent of $h$ such that 
\[\frac{\cR_{L}-\cR_{\vartheta L}}{\cR_{\vartheta L}}\ge \frac{C(n)}{\vartheta^2}, \qquad\text{for some $\vartheta \in (0,1)$ and every $L\in(0,L_0]$}\]
then for every $L\in(0,L_0]$
\[\sigma_k(M)\ge \frac{C(n)}{4k\cR_L}\sup_{\mathfrak{U}_L}\mu_k(\Omega)\,L^2.\]
\end{proposition}
The assumption implies that the resistance along 
$[\vartheta L,L]$ is bounded below by a fixed multiple 
of the resistance along $[0,\vartheta L]$, preventing any concentration of the resistance along~$(\vartheta L, L]$. 
\begin{remark}
   When the  boundary collar has  a product structure, i.e. $h\equiv 1$, or  more generally $h$ is non-increasing {on $(0,\injM]$}, then one can take $C(n)=\frac{1}{4}$  and $\vartheta=\half$ in the above proposition and get 
\begin{equation}\label{eq:product-lower-bound}
\sigma_k(M)\ge \frac{1}{16k} \,\mu_k(\Omega) L,\qquad 1\le k\le b-1,
\end{equation}
{for any $L\in (0,\injM]$ without restricting the range of $L$ to $(0,L_k(\injM)]$.} \medskip

In the hyperbolic setting if $L_\beta$ tends to infinity, where $L_\beta$ is defined in Examples \ref{ex:hypsurface} and~\ref{ex:hyperbolic-n-mfld},  then for any given $\vartheta\in(0,1)$, $\frac{\cR_{L_\beta}}{\cR_{L_{\vartheta\beta}}}\to 1$ and therefore, $\frac{\cR_{L}-\cR_{\vartheta L}}{\cR_{\vartheta L}}$, $L\in [0,L_\beta]$
does not satisfy a uniform lower bound strictly larger than $1$. However, if we take $L_0=\min\{1,L_\beta\}$, then the assumption of Proposition \ref{cor:k-th stek lower bd for warped bdry} is met for  $\vartheta=\half$ and $C(n)=\frac{1}{4}\cosh(1)^{1-n}$, {and we can recover Corollary \ref{cor:hyperbolic}. Notice that    $L_0$ and $L_k(L_\beta)$ up to possibly a multiplicative universal constant are the same.}
\end{remark}

\subsection{{Pinched Negatively Curved Manifolds}} Hyperbolic manifolds with totally geodesic boundary belong to the larger class of pinched negatively curved manifolds. Although such manifolds do not necessarily have a warped product structure in a boundary collar, the above bounds also hold in this setting. Moreover, for a suitable choice of $\Omega\in \mathfrak{U}_{L_{\beta,k}}$, we can obtain a lower bound for \(\mu_k(\Omega)\) and recover results analogous to those in \cite{Per25,HMP25} for surfaces and in \cite{BBHM} for higher dimensions. \medskip

We first state our result for pinched negatively curved surfaces.  
Let $(M,g)$ be a compact Riemannian surface with geodesic boundary and with  the sectional curvature in $[-1,-\delta^2]$, where
$\delta\in(0,1]$.  We denote by $\gamma$ the
genus of $M$, by $b$ the number of boundary components, and by $\beta$
the maximum length of  boundary components. 
\begin{definition}[\cite{HMP25,Per25}]
Let $\bC_k$ be the set of multi-geodesics (a union of disjoint simple closed geodesics) not intersecting $\partial M$, and dividing $M$ into $k+1$ connected components each containing at least one connected component of $\partial M$. Define
\begin{align*}
\ell_k ( g) := \inf_{\bc \in \bC_k} |\bc|_g,
\end{align*}
where $|\bc|_{ g}$ is the length of the {multi-geodesic} $\bc$. When $\bC_k=\emptyset$, we set $\ell_k( g)=\infty$. 
\end{definition}
In \cite[Theorem 1.1]{HMP25}, the authors showed that $\ell_k(g)<\infty$ for $1\le k\le b-1$ when $g\ge1$, and for $1\le k\le b-3$ when $g=0$.  

\begin{theorem}\label{prop: HMP bound}
    Let  $(M,g)$ be a compact Riemannian surface with geodesic boundary and with sectional curvature in $[-1,-\delta^2]$, $\delta\in(0,1]$. Then
    \[ \sigma_k(M,g)\ge \frac{\delta C(b,\gamma)}{k}\min\left\{\frac{1}{e^{\beta/2\delta}},\frac{\ell_k(g)}{e^{\beta/\delta}}\right\},\qquad 1\le k\le b-1.
\]
\end{theorem} 
Theorem \ref{prop: HMP bound} gives a similar lower bound to that in~\cite{HMP25} as a consequence of inequality~\eqref{eq:lower-bound-hyperbolic-surfaces}, together with an extension of Dodziuk--Randol's estimate for $\mu_k(\Omega)$; see Lemma~\ref{lem: DR mu k}. {In \cite{HMP25}, the authors introduce a  thick-thin decomposition tailored to the Steklov setting and adapt the proof of Dodziuk and Randol accordingly. We avoid this step and instead use Dodziuk-Randol result directly for Neumann eigenvalues.\\}

We now use the above theorem to conclude that the rate of
$\sup_{\Omega\in\mathfrak{U}_C}\sigma_k(\Omega)$
in Theorems~\ref{thm:general result} and
\ref{thm: general result II} is sharp.
\begin{remark}\label{rem:optimality of muk} Let $M$ be a hyperbolic surface of genus $\gamma\ge 1$.
 Assume that  and all boundary
components have equal length $\beta=1$. For any $1\le k\le b-1$, let
$g_j$ be a sequence of hyperbolic metrics on $M$ such that
$\ell_k(g_j)\to 0$.
For $n$ large enough, by \cite{Per25,HMP25}, we have
\[
\sigma_k(M,g_j)\le C_1(b,\gamma)\ell_k(g_j),
\qquad 1\le k\le b-1.
\]
As an immediate consequence of the proof of Theorem \ref{prop: HMP bound} and Lemma \ref{lem: DR mu k}, we can choose
$\Omega_n\subset (M,g_j)$ such that
\[
\mu_k(\Omega_n,g_j)\ge C_2(b,\gamma)\ell_k(g_j).
\]
Therefore, for $n$ large enough,
\[
C_3(b,\gamma)\sup_{\Omega\in\mathfrak{U}_C}\mu_k(\Omega,g_j)
\le
\sigma_k(M,g_j)
\le
C_4(b,\gamma)\sup_{\Omega\in\mathfrak{U}_C}\mu_k(\Omega,g_j),
\]
where $C=L_{k,1}$. This shows, in
particular, that the rate of
$\sup_{\Omega\in\mathfrak{U}_C}\sigma_k(\Omega)$
in Theorems~\ref{thm:general result} and
\ref{thm: general result II}.
Moreover, one can construct a family of metrics for which
$\sup_{\Omega\in\mathfrak{U}_C}\mu_k(\Omega,g_j)\to 0,$
while
$\lim_{n\to\infty}
\sup_{\Omega\in\mathfrak{U}_C}\mu_{k+1}(\Omega,g_j)>0.
$
\end{remark}
Motivated by the example above and the discussion in the introduction, it is
natural to ask whether one can have an upper bound of the following form for some $L\in (0,\injM)$ 
\[
\sigma_k(M,g)\le
{\frac{C(n)}{\cR_L}}\sup_{\Omega\in\mathfrak{U}_L}\mu_k(\Omega,g)L^2,
\qquad 1\le k\le b-1,
\]
{We construct a family of Riemannian manifolds with fixed boundary and a product structure near the boundary whose normal injectivity radius remains uniformly bounded. This shows that such an upper bound cannot hold in general.}

\begin{theorem}\label{thm:counterexample upper bound}
    There exists a sequence of compact Riemannian manifolds $(M_j^n,g_j)$ with
boundary whose boundary collars are isometric to
$[0,1]\times \partial M_j$, and whose normal injectivity radii are uniformly
bounded below, such that, for every $L$ smaller than the normal injectivity
radius of $\partial M_j$,
\[
    \frac{\sigma_1(M_j,g_j){\cR_L}}
    {\sup_{\Omega\in\mathfrak{U}_L}\mu_1(\Omega)L^2}
    \longrightarrow \infty
    \qquad \text{as } j\to\infty .
\]
\end{theorem}

Let us return to the case of pinched negatively curved $n$-manifold, $n\ge3$. As in the case of surfaces, combining inequality
\eqref{eq:lower-bound-hyperbolic-n-manifolds} with a lower bound for
$\mu_k(M)$ yields a lower bound for the spectral gap of pinched negatively
curved $n$-manifolds, $n\ge 3$. This gives a similar result to
\cite[Theorem 1.3]{BBHM}, and improves their bound
when $\beta\gg 1$. 

\begin{theorem}\label{prop: BBHM new}
    Let $(M^n,g)$, $n \ge 3$, be a compact connected Riemannian manifold whose sectional curvatures lie in $[-1,-\delta^2]$ for some $\delta \in (0,1]$. Assume that $\partial M$ has $b$ totally geodesic connected components $\Sigma_j$, $1 \le j \le b$, and set $\beta := \max_{1 \le j \le b} |\Sigma_j|$. Then
    \[\sigma_1(M)\ge\frac{C(n,\delta)}{|M|^2\,\beta^{\frac{2}{\delta(n-2)}}}.\]
\end{theorem}
Although we use the tubular neighbourhood theorem from \cite{BBHM} in the proof of Theorem \ref{prop: BBHM new}, we avoid the technical adaptation of the Dodziuk--Randol argument to the Steklov setting altogether and only apply the classical lower bound obtained by Schoen \cite{Schoen82} to bound $\mu_1(M)$.

\begin{remark}
   Theorems \ref{prop: HMP bound} and \ref{prop: BBHM new} remain valid for pinched negatively curved manifolds with  interior cusps. The proof is unchanged, since the trace inequality used in the proof of Theorem \ref{thm: general result II} continues to hold in this setting. Moreover, the results of Schoen--Wolpert--Yau \cite{SWY80} for surfaces and Schoen \cite{Schoen82} for higher dimensions extend to the non-compact setting; see \cite{DR86,Dod,DRS87}.
\end{remark}

\section{{Proofs for Warped Boundary Collars and Pinched Negative Curvature}}\label{sec: proofs in hyperbolic setting}
{In this section, we prove the results stated in the previous section. We begin with the proof of Proposition \ref{cor:k-th stek lower bd for warped bdry}. }
\begin{proof}[\normalfont\bf Proof of Proposition \ref{cor:k-th stek lower bd for warped bdry}]
 It is enough to show that
$\mu_k(\Omega)L^2$ is uniformly bounded above by a constant $C_1(n)$ for every $L\in[0,L_0]$ and
$\Omega\in \mathfrak{U}_L$. Then, in Theorem \ref{thm:general result}, we can replace
$1$ by $C_1^{-1}(n)\mu_k(\Omega)L^2$.  Let $\lambda_k^{DN}(\mathcal{C}_L)$ denote
the $k$th mixed Dirichlet--Neumann eigenvalue of $\mathcal{C}_L$, with
Dirichlet condition on the interior boundary. By separation of variables, the
eigenfunctions associated with $\lambda_j^{DN}(\mathcal{C}_L)$, $1\le j\le b$,
are radial. Hence
\[
\mu_k(\Omega)\le \lambda_{k+1}^{DN}(\mathcal{C}_L)=\cdots
=\lambda_1^{DN}(\mathcal{C}_L),\qquad 1\le k\le b-1.
\] We claim that
\begin{equation}\label{eq:lambda DN bound wp}
\lambda_1^{DN}(\mathcal{C}_L)
\le \frac{1}{r^2}\,\frac{\int_0^{r} h(t)^{1-n}\,dt}
{\int_{r}^L h(t)^{1-n}\,dt},\quad\forall r\in (0,L).
\end{equation}
{Therefore, taking $r=\vartheta L$, we get 
\begin{equation}\label{eq:condition-on-h}
\lambda_1^{DN}(\mathcal{C}_L)L^2\le\frac{1}{\vartheta^2} \frac{\int_0^{\vartheta L} h(t)^{1-n}\,dt}{\int_{\vartheta L}^L h(t)^{1-n}\,dt} =\frac{1}{\vartheta^2} \frac{\cR_{\vartheta L}}{\cR_{L}-\cR_{\vartheta L}}\le \frac{1}{C(n)}
\end{equation}
where the last inequality follows from the assumption, and this completes the proof.\\
 It remains to prove} the claimed inequality~\eqref{eq:lambda DN bound wp}. For any $r\in(0,L)$, define the test function $\phi_r(t)=1$ for $0\le t\le r$ and
\[
\phi_r(t)=
\dfrac{\displaystyle\int_t^L h(s)^{1-n}\,ds}
{\displaystyle\int_r^L h(s)^{1-n}\,ds},
\qquad r\le t\le L.
\]
Calculating its Rayleigh quotient gives
\begin{align*}
\lambda_1^{DN}(\mathcal{C}_L)
&\le
\frac{\int_0^L h(t)^{n-1}|\phi_r'(t)|^2\,dt}
{\int_0^L h(t)^{n-1}\phi_r(t)^2\,dt}\\
&\le
\frac{1}
{\left(\int_r^L h(t)^{1-n}\,dt\right)
 \left(\int_0^r h(t)^{n-1}\,dt\right)}\\
&\le
\frac{\int_0^r h(t)^{1-n}\,dt}
{r^2\int_r^L h(t)^{1-n}\,dt},
\end{align*}
where the last inequality follows from the Cauchy-Schwarz inequality.
\end{proof}}

{We  proceed with the proof of our results for pinched negatively curved $n$-manifolds starting with the case of surfaces, $n=2$.}

\begin{proof}[\normalfont\bf Proof of Theorem \ref{prop: HMP bound}] {The two main ingredients {of the proof} are Theorem \ref{thm: general result II}, more precisely Corollary \ref{cor:hyperbolic}, and a version of Dodziuk--Randol's estimate for the $k$th Neumann eigenvalue of pinched negatively curved surfaces. }

{As shown in the proof of Theorem~3.8 in
\cite{HMP25},} using the normalised Ricci
flow and a generalisation of the Ahlfors--Schwarz argument, one can show that $g=f\bar g$, where $\bar g$ is a hyperbolic metric on
$M$, and that the conformal factor $f$ satisfies
$1\leq f\leq \frac{1}{\tau^2}$. Hence,
\begin{equation}\label{eq: conformal change bounds}
{\tau}\,\sigma_k(M,\bar g)\leq \sigma_k(M,g)\leq \sigma_k(M,\bar g).
\end{equation}
Therefore, we first assume that $(M,g)$ is a hyperbolic surface and show
\begin{equation}\label{prop: HMP counterpart}
 \sigma_k(M,g)\ge \frac{C(b,\gamma)}{k}\min\left\{\frac{1}{e^{\beta/2}},\frac{\ell_k(g)}{e^\beta}\right\},\qquad 1\le k\le b-1.
\end{equation}
The main idea is to obtain a lower bound for $\mu_k(\Omega)$, for a suitable subfamily of domains $\Omega\in\mathfrak{U}_L$, in terms of the desired geometric quantities, and then apply Corollary \ref{cor:hyperbolic}. Specifically, we restrict our attention to those subdomains $\Omega\in\mathfrak{U}_L$ whose boundary is geodesic. For closed hyperbolic surfaces, such lower bounds were established in \cite{Schoen82,DR86}. Moreover, the argument in \cite{DR86} can be adapted, with only minor modifications, to hyperbolic surfaces with geodesic boundary.
 We begin by introducing the necessary definitions.\medskip 

Let $\Omega$ be a connected hyperbolic surface whose boundary (if it is nonempty) consists of a
collection of closed geodesics, and let $\widetilde{\bC}_k(\Omega)$ be the set
of multi-geodesics in $\Omega$ which do not contain any component of
$\partial\Omega$ and which divide $\Omega$ into $k+1$ connected components.
Define
\begin{align*}
\widetilde\ell_k(\Omega) := \inf_{\bc \in \widetilde{\bC}_k(\Omega)} |\bc|_g,
\end{align*}
where $|\bc|_g$ denotes the length of the multi-geodesic $\bc$. If
$\widetilde{\bC}_k(\Omega)=\emptyset$, we set
$\widetilde\ell_k(\Omega)=\infty$. This definition for closed hyperbolic surfaces coincides with the definition introduced by Schoen--Wolpert--Yau \cite{SWY80}.
\begin{remark} If
$\Omega_1\subset\Omega_2$, then
$\widetilde\ell_k(\Omega_2)\le \widetilde\ell_k(\Omega_1)$. 
In particular, when $\Omega \subset M$,
$\bC_k\subset\widetilde{\bC}_k(M)$, and hence
$\widetilde\ell_k(M)\le \ell_k(g)$. 
\end{remark}
\noindent We now state a version of the Dodziuk--Randol estimate \cite[Lemma~4]{DR86} for hyperbolic manifolds with geodesic boundary. Note that in \cite{HMP25}, the authors first established a lower bound for $\sigma_1(M)$ by  adapting Dodziuk-Randol's argument. We bypass this step and use Dodziuk-Randol's result for first positive Neumann  eigenvalue. 
\begin{lemma}\label{lem: DR Neumann}
Let  $\Omega$ by a hyperbolic surface with non-empty geodesic boundary. Then
    \[\mu_1(\Omega)\ge C(\gamma,b)\min\{e^{-\beta_\Omega/2},\widetilde\ell_1(\Omega)\},\]
    where $\beta_\Omega$ is the maximum length of the boundary components of $\Omega$.
\end{lemma}
\begin{proof}
   Let $\widetilde\Omega$ be the double of $\Omega$, obtained by gluing two copies of $\Omega$ along its geodesic boundary.  By the proof of \cite[Lemma~4]{DR86}, one has
\[
   \mu_1(\Omega)\ge C_1(\gamma,b)\min\{1,\widetilde\ell_1^*(\widetilde\Omega)\},
\]
where $\widetilde\ell_1^*(\widetilde\Omega):=
\min\{|\bc|_g:\bc\in \widetilde\bC_k(\widetilde\Omega)\text{ and }\bc\cap\partial\Omega\text{ is either empty or finite}\}$.
If $\widetilde\ell_1^*(\widetilde\Omega)=\widetilde\ell_1(\Omega)$, then the desired
estimate follows immediately. Otherwise, $\widetilde\ell_1^*(\widetilde\Omega)$ is
realised by a multigeodesic whose restriction to $\Omega$ contains at least one
geodesic arc orthogonal to a boundary component of $\partial\Omega$. The length of such an arc is bounded from below by twice the width of the collar about that boundary geodesic. Hence, by the collar theorem, it is bounded below by $2L_{\beta_\Omega}$, and by \eqref{eq:lower bound for L} it is bounded below by $Ce^{-\beta_\Omega/2}$, and the result follows.
\end{proof}
\begin{lemma}\label{lem: DR mu k}
Let $\Omega$ be a connected compact hyperbolic surface with geodesic boundary, which is a
union of closed simple geodesics, we have
\[
\mu_{k}(\Omega)\ge C(\gamma,b)
\min\left\{
e^{-\frac{\beta_\Omega}{2}},
\widetilde\ell_{k,\beta_\Omega}(\Omega)
\right\},
\]
where $\beta_\Omega$ is the maximum length of the boundary components of $\Omega$ and
\[
\widetilde\ell_{k,\beta_\Omega}(\Omega):=
\inf\left\{
|\bc|_g:
\bc=\alpha_1\sqcup\cdots\sqcup\alpha_p\in\widetilde\bC_k(\Omega),
\quad
|\alpha_j|\le \max\{\beta_\Omega,\operatorname{arsinh}1\},
\quad 1\le j\le p
\right\}.
\]
We set $\widetilde\ell_{k,\beta_\Omega}(\Omega)=\infty$ if the set is empty.
\end{lemma}
The proof of Lemma \ref{lem: DR mu k} follows  from \cite[Lemma~4]{DR86}. An adaptation of the same argument to the Steklov problem is used in the proof of Theorem~3.5 in
\cite{HMP25}.
\begin{proof}[Proof of Lemma \ref{lem: DR mu k}]
  We first consider the case that $\widetilde \ell_{k,\beta_\Omega}(\Omega)<\infty$.   Let $\bc=\alpha_1 \sqcup \dots \sqcup \alpha_p{\in \widetilde\bC_k(\Omega)}$ be an admissible multigeodesic realising $\widetilde \ell_{k,\beta_\Omega}(\Omega)$. Thus 
  \[
|\alpha_j|\le \max\{\beta_\Omega,\operatorname{arsinh}1\},
\qquad 1\le j\le p.
\]
  Then at least  one of the $p$ components of $\bc$ must be of length $\geq \frac{{\widetilde\ell_{k,\beta_\Omega}(\Omega
  )}}{p}$; after re-numerating, assume 
  $$\max\{\beta_\Omega,\arsinh1\}\ge|\alpha_{p}|\ge \frac{{\widetilde\ell_{k,\beta_\Omega}(\Omega
  )}}{p}\ge \frac{C_1}{\gamma+b}{\widetilde\ell_{k,\beta_\Omega}(\Omega
  )}.$$ We decompose $\Omega$ into $k$ components $\Omega_1,\ldots,\Omega_k$  by removing from $\Omega$ all the geodesics of $\bc$ except $\alpha_{p}$.  By Lemma~\ref{lem: DR Neumann}, applied to each $\Omega_j$, and using $\beta_{\Omega_j}\le \max\{\beta_\Omega,\operatorname{arsinh}1\},$
we obtain
  \[\mu_k(\Omega)\ge \min_j\mu_1(\Omega_j)\ge C_2(\gamma,b)\min\{e^{-\frac{\beta_\Omega}{2}},\widetilde\ell_1(\Omega_j)\}.\] 
  Indeed, the inequality
$\beta_{\Omega_j}\le \max\{\beta_\Omega,\operatorname{arsinh}1\}$ follows because
the boundary components of $\Omega_j$ are either boundary components of $\Omega$
or components of the admissible multigeodesic $\bc$.

there is nothing to prove for that
component. Otherwise, let $\bc_j$ be a multigeodesic in $\Omega_j$ realising
$\widetilde\ell_1(\Omega_j)$.   Suppose that
\[
|\bc_j|<|\alpha_p|\le \max\{\beta_\Omega,\operatorname{arsinh}1\}.
\]
Hence, $
\widetilde\bc =\alpha_1\sqcup\cdots\sqcup\alpha_{p-1}\sqcup\bc_j
$
is admissible for $\widetilde\ell_{k,\beta_\Omega}(\Omega)$. Moreover, $
|\widetilde\bc|<|\bc|$,
which contradicts the choice of $\bc$. Thus, for all $j$, $\widetilde\ell_1(M_j)\ge |\alpha_p|\ge \frac{\widetilde\ell_{k,\beta_\Omega}(\Omega)}{p}$. It follows that
\[
\mu_k(\Omega)\ge C_3(\gamma,b)
\min\left\{
e^{-\frac{\beta_\Omega}{2}},
\widetilde\ell_{k,\beta_\Omega}(\Omega)
\right\},
\]
as required.\\ 

We now consider the case when $\widetilde\ell_{k,\beta_\Omega}(\Omega)=\infty$. Take $1\leq s< k$ be the largest $s$ such that  $\widetilde\ell_{s,\beta_\Omega}(\Omega)<\infty$.
If no such $s$ exists, then $\widetilde\ell_{1,\beta_\Omega}(\Omega)=\infty$. 
This implies that every multigeodesic contributing to $\widetilde\ell_1(\Omega)$
has at least one component of length larger than
$\max\{\beta_\Omega,\operatorname{arsinh}1\}$. Hence
\[
\widetilde\ell_1(\Omega)
\ge
\max\{\beta_\Omega,\operatorname{arsinh}1\}
\ge
C_4 e^{-\beta_\Omega/2}.
\]
Note that it is important here to take the maximum of a fixed constant and
$\beta_\Omega$, since the inequality does not hold when $\beta_\Omega$ is small.

\noindent By Lemma \ref{lem: DR Neumann}, we conclude
\[
\mu_k(\Omega)\ge\mu_1(\Omega)\ge C_5(b,\gamma)e^{-\beta_\Omega/2}. 
\]

Otherwise, choose an admissible multigeodesic $\bc\in \widetilde\bC_s(\Omega)$ realising $\widetilde\ell_{s,\beta_\Omega}(\Omega)$. Cutting $\Omega$ along
$\bc$ decomposes it into $s+1$ connected components $
\Omega_1,\ldots,\Omega_{s+1}$. Suppose that for some $j$ we have $\widetilde\ell_{1,\beta_\Omega}(\Omega_j)<\infty$, and let $\bc_j$ be a multi-geodesic in $\Omega_j$ realising $\widetilde\ell_{1,\beta_\Omega}(\Omega_j)$. Then $\bc\sqcup \bc_j$ would belong to $\widetilde\bC_{s+1}(\Omega)$, with $\widetilde\ell_{s+1,\beta_\Omega}(\Omega)<\infty$, contradicting the maximality of $s$. Therefore $
\widetilde\ell_{1,\beta_\Omega}(\Omega_j)=\infty
$ for every $j$. As above, this implies
\[
\widetilde\ell_1(\Omega_j)
\ge
\max\{\beta_\Omega,\operatorname{arsinh}1\}
\ge
C_6 e^{-\beta_\Omega/2}.
\]
Moreover, since the boundary components of each $\Omega_j$ are either boundary
components of $\Omega$ or components of the admissible multigeodesic $\bc$, we have
\[
\beta_{\Omega_j}\le \max\{\beta_\Omega,\operatorname{arsinh}1\}.
\]
Thus, by Lemma~\ref{lem: DR Neumann},
\[
\mu_k(\Omega)\ge \mu_{s+1}(\Omega)\ge\min_j\mu_1(\Omega_j)\ge C_7(b,\gamma)e^{-\beta_\Omega/2}.
\]
\end{proof}
We now proceed with the proof of inequality \eqref{prop: HMP counterpart}.
   If $\widetilde\ell_{k,\beta}(M)=\infty$ or
$\widetilde\ell_{k,\beta}(M)=\ell_k(g)$, we take $\Omega=M$ in inequality \eqref{eq:lower-bound-hyperbolic-surfaces}.
We then apply Lemma~\ref{lem: DR mu k} to $\mu_k(M)$ and conclude 
\[ \sigma_k(M,g)\ge \frac{C(b,\gamma)}{k}\min\left\{\frac{1}{e^{\beta/2}},\frac{\ell_k(g)}{e^\beta}\right\},\qquad 1\le k\le b-1.
\]

\noindent Note that when $\widetilde\ell_{k,\beta}(M)<\infty$, $\widetilde\ell_{k,\beta}(M)\le \ell_k(g)$. Thus, it  remains to consider the case
\[
\widetilde\ell_{k,\beta}(M)<\ell_k(g).
\]
Let $\bc$ be an admissible multigeodesic realising
$\widetilde\ell_{k,\beta}(M)$. We consider the connected components of
$M\setminus\bc$ and discard every connected component that does not contain any
boundary component of $M$. We denote the resulting domain by $\Omega_0$. Notice
that $\beta_{\Omega_0}=\beta$.
If $\widetilde\ell_{k,\beta}(\Omega_0)=\infty
\quad\text{or}\quad
\widetilde\ell_{k,\beta}(\Omega_0)=\ell_k(g),
$
we take $\Omega=\Omega_0$. Otherwise, we repeat the same procedure. Since at
each step at least one connected component is discarded, the process terminates
after finitely many steps. 

Now the statement of Theorem \ref{prop: HMP bound} follows immediately from inequality~\eqref{prop: HMP counterpart} and the bound on the conformal factor after conformally deforming $g$ to a hyperbolic metric.
\end{proof}
\begin{proof}[\normalfont\bf Proof of Theorem \ref{prop: BBHM new}]
{Given $n\ge 3$},
    let $V_{n-1,\tau}(r)$ denotes the volume of a ball of radius $r$ in the space form  $\mathbb{H}^{n-1}({-\tau^2})$ of sectional curvature $-\tau^2$, and let $\rho(r):=\log\coth\frac{r}{2}$. Then by the tubular neighbourhood theorem \cite[Theorem 1.1]{BBHM},  we can take
$$L_0=L_{\beta,\tau}:=\frac{1}{2}(V_{n-1,\tau}\circ \rho)^{-1}(\beta)$$ in Theorem \ref{thm:general result}. By \cite[Lemma 3.5, Remark 3.6]{BBHM}, there exists a positive constant $C_1(n,\tau)$ such that 
\begin{equation}\label{eq:Lbeta lowerbound}
    L_{\beta,\tau}\ge C_1(n,\tau)\beta^{-\frac{1}{\tau(n-2)}}.
\end{equation}
Moreover, by a version of the G\"unter--Bishop comparison theorem for tubes
(see \cite[Lemmas 8.22 and 8.24]{G03}; see also \cite{BBHM}), we have

$$\cosh^{n-1}(\tau t)\le\varphi_j(x,t)\le \cosh^{n-1}(t), \qquad t\in(0,L_{\beta}),\quad x\in \Sigma_j.$$
\begin{equation}\label{eq:Al upper bound}
    \cR_L\le\int_0^L\cosh^{1-n}(\tau t)dt\le \frac{\tau(n-1)}{2^{n-1}};
\end{equation}
We now repeat the calculation from Example \ref{ex:hyperbolic-n-mfld} and compute
$L_1(L_{\beta,\tau})
=\cR^{-1}\left(\frac{\cR_{L_{\beta,\tau}}}{4}\right)$. Since
$\cosh^{1-n}(\tau t)\leq 1$, we have $\cR_L\leq L$, and hence
\[
L_1(\beta,\tau)
=
\cR^{-1}\left(\frac{\cR_{L_{\beta,\tau}}}{4}\right)
\geq
\frac{\cR_{L_{\beta,\tau}}}{4}.
\]
Moreover,
\[
\Phi(L_{\beta,\tau})
\ge
\int_0^{L_{\beta,\tau}}\cosh^{1-n}t\,dt
\geq
\int_0^{\min\{L_{\beta,\tau},1\}}\cosh^{1-n}t\,dt
\geq
\cosh(1)^{1-n}\min\{L_{\beta,\tau},1\}.
\]
Therefore
\[
L_1(\beta,\tau)
\geq
\frac{\cosh(1)^{1-n}}{4}\min\{L_\beta,1\}\ge {C_2(n,\tau)}\beta^{-\frac{1}{\tau(n-2)}}.
\]
Replacing $\cR_L$ and $L$ in Corollary \ref{cor:k-th stek lower bd for warped bdry} by the bounds obtained above, we obtain
\begin{equation*}
    \sigma_k(M) \ge \frac{C_3(n,\tau)}{{\beta}^{\frac{2}{\tau(n-2)}}}\sup_{\Omega\in{\mathfrak{U}_L}}\mu_1(\Omega).
\end{equation*}
It remains to give a lower bound for $\mu_1(\Omega)$. Take $\Omega=M$ and   let $\widetilde M$ denote the double of $M$, obtained by gluing two copies of $M$ along $\partial M$. Then $\widetilde M$ is a smooth closed hyperbolic $n$-manifold (since $\partial M$ is totally geodesic) and $|\widetilde M|=2|M|$. Moreover, $
\lambda_1(\widetilde M)\le\mu_1(M),$
since the spectrum of $\widetilde M$ splits into the Neumann and Dirichlet spectra of $M$. We  use  Schoen's lower bound \cite{Schoen82},
\[
\lambda_1(\widetilde M)
\ge \frac{C_4(n,\tau)}{|M|^2}
\]
to conclude.
\end{proof}

\section{Proof of Theorems \ref{thm: manifolds with big mu1} and \ref{thm:counterexample upper bound} {via Family of Examples}}\label{sec: proof of examples}
In this section, we construct two family of examples. The first family shows that, without any restriction on the range of $L\in (0,\injM]$, we cannot remove the minimum in the statement of Theorem \ref{thm:general result}; this proves Theorem \ref{thm: manifolds with big mu1}. The second family shows $\sup_{\Omega\in \mathfrak{U}_L}\mu_1(\Omega)$ may decay much faster than $\sigma_1(M)$ proving Theorem \ref{thm:counterexample upper bound}. Both families of examples have warped product structure in a boundary collar.

\begin{proof}[\normalfont\textbf{Proof of Theorem \ref{thm: manifolds with big mu1}}]
    Fix \(L>0\) and  let \(\Sigma_j\) be a sequence of closed
\((n-1)\)-dimensional manifolds with 
        $\lambda_1(\Sigma_j)\to \infty$ in a rate to be specified in \eqref{eq: lambda1 growth}. 
        When the dimension is at least 3 we can further assume that $|\Sigma_j|=1$.  Let $a_j\to\infty$, and consider $M_j=[-L,L]\times \Sigma_j$ with the warped product metric
\[
        g_j=dt^2+h_j(t)^2g_{\Sigma_j},
        \qquad 
        h_j(t)=e^{a_j(L-|t|)} .
\]
We shall see that this is the desired family of manifolds.  
We take the radial Steklov test function which is $1$ on one boundary
component, $-1$ on the other and $0$ in the middle and harmonic in between. Note that \[
\Delta_g u
= -\frac{1}{h(t)^{n-1}}
\partial_t\!\left(
h(t)^{n-1}\partial_t u
\right)
+
\frac{1}{h(t)^2}\Delta_\Sigma u .
\] Hence, the harmonic radial test function satisfies $h_j^{\,n-1}u'= c$
for some constant $c$. Integrating from $0$ to $t\in(0,L]$, we get
\[
u(t)-u(0) = c\int_0^t h_j(s)^{1-n}\,ds.
\]
By assumption $u(0)=1$ and $u(L)=0$. 
Then $c=-1/\cR_{L,j}$ and
\[
u'(t)=\frac{-1}{\cR_{L,j}}h_j(t)^{1-n}.
\]
Therefore
\begin{align*}
        \sigma_1(M_j)&\le\frac{\int_{-L}^{L}\int_{\Sigma_j} |u'(t)|^2 h_j^{n-1}\,dA_{\Sigma_j}dt}{2\int_{\Sigma_j} 1 \,dA_{\Sigma_j}dt}
       \le \frac{1}{\cR_{L,j}}.
\end{align*}
On the other hand, the Neumann Rayleigh quotient on $M_j$ is
\[
\mu_1(M_j)
=
\inf_{\substack{u\not\equiv 0\\
\int_{M_j}
u =0}}\frac{
\displaystyle
\int_{-L}^{L}\int_{\Sigma_j}
\left(
|\partial_t u|^2
+
h_j(t)^{-2}|\nabla_{\Sigma_j}u|^2
\right)
h_j(t)^{n-1}\,dv_{\Sigma_j}\,dt
}{
\displaystyle
\int_{-L}^{L}\int_{\Sigma_j}
u^2 h_j(t)^{n-1}\,dv_{\Sigma_j}\,dt
}.
\]
We decompose {any admissible $u$ as follows.}
\[
u(t,y)=\overline u(t)+u^\perp(t,y),
\]
where
\[
\overline u(t)
=
\frac{1}{|\Sigma_j|}
\int_{\Sigma_j}u(t,y)\,dv_{\Sigma_j},
\qquad
\int_{\Sigma_j}u^\perp(t,y)\,dv_{\Sigma_j}=0.
\]
Then
\[
\int_{-L}^{L}\int_{\Sigma_j}
u^2 h_j(t)^{n-1}\,dv_{\Sigma_j}\,dt
=
\int_{-L}^{L}\int_{\Sigma_j}
\overline u(t)^2 h_j(t)^{n-1}\,dv_{\Sigma_j}\,dt
+
\int_{-L}^{L}\int_{\Sigma_j}
(u^\perp)^2 h_j(t)^{n-1}\,dv_{\Sigma_j}\,dt.
\]
For the radial part, since
\[
h_j(t)^{n-1}
=
e^{(n-1)a_jL}e^{-(n-1)a_j|t|},
\]
the constant factor \(e^{(n-1)a_jL}\) cancels in the quotient. Thus the relevant weight is
\[
e^{-(n-1)a_j|t|}.
\]
{Let $b_j=(n-1)a_j$. Since
\[
\int_{-L}^{L}\overline{u}(t)e^{-b_j|t|}\,dt=0,
\]
we have
\[
\int_{-L}^{L}\overline u(t)^2 e^{-b_j|t|}\,dt
\le
\int_{-L}^{L}\bigl(\overline u(t)-\overline u(0)\bigr)^2
e^{-b_j|t|}\,dt .
\]
Set $v(t)=\overline u(t)-\overline u(0)$.  On $[0,L]$,
integration by parts gives
\[
\int_0^L v(t)^2 e^{-b_jt}\,dt
=
-\frac{v(L)^2e^{-b_jL}}{b_j}
+
\frac{2}{b_j}\int_0^L v(t)v'(t)e^{-b_jt}\,dt .
\]
Therefore, by the Cauchy-Schwarz inequality,
\[
\int_0^L v(t)^2 e^{-b_jt}\,dt
\le
\frac{4}{b_j^2}
\int_0^L |v'(t)|^2e^{-b_jt}\,dt .
\]
The same inequality holds on $[-L,0]$. Therefore,
\[
\int_{-L}^{L}\overline u(t)^2e^{-(n-1)a_j|t|}\,dt
\le
\frac{4}{(n-1)^2a_j^2}
\int_{-L}^{L}|\overline u'(t)|^2e^{-(n-1)a_j|t|}\,dt,
\]
and }
\[
\frac{
\displaystyle
\int_{-L}^{L}|\overline u'(t)|^2 h_j(t)^{n-1}\,dt
}{
\displaystyle
\int_{-L}^{L}\overline u(t)^2 h_j(t)^{n-1}\,dt
}
\geq
\frac{(n-1)^2a_j^2}{4}.
\]
{The above inequality can be viewed as a special case of a general weighted Hardy inequality,} see \cite[Theorem 1]{Muc72}.\\
For the non-radial part, the Poincaré inequality on $\Sigma_j$ gives
\[
\int_{\Sigma_j}|\nabla_{\Sigma_j}u^\perp|^2\,dv_{\Sigma_j}
\geq
\lambda_1(\Sigma_j)
\int_{\Sigma_j}(u^\perp)^2\,dv_{\Sigma_j}.
\]
Hence
\[
\begin{aligned}
&\int_{-L}^{L}\int_{\Sigma_j}
h_j(t)^{-2}|\nabla_{\Sigma_j}u^\perp|^2
h_j(t)^{n-1}\,dv_{\Sigma_j}\,dt \\
&\qquad\geq
\lambda_1(\Sigma_j)
\int_{-L}^{L}\int_{\Sigma_j}
h_j(t)^{-2}(u^\perp)^2
h_j(t)^{n-1}\,dv_{\Sigma_j}\,dt.
\end{aligned}
\]
Since $h_j(t)\leq e^{a_jL}$, we have
$ h_j(t)^{-2}\geq e^{-2a_jL}$.
Therefore
\[
\begin{aligned}
&\int_{-L}^{L}\int_{\Sigma_j}
h_j(t)^{-2}|\nabla_{\Sigma_j}u^\perp|^2
h_j(t)^{n-1}\,dv_{\Sigma_j}\,dt \\
&\qquad\geq
\lambda_1(\Sigma_j)e^{-2a_jL}
\int_{-L}^{L}\int_{\Sigma_j}
(u^\perp)^2 h_j(t)^{n-1}\,dv_{\Sigma_j}\,dt.
\end{aligned}
\]
Combining the radial and non-radial estimates, we obtain
\[
\mu_1(M_j)
\geq
\min\left\{
\frac{(n-1)^2a_j^2}{4},
\,
\lambda_1(\Sigma_j)e^{-2a_jL}
\right\}.
\]
In particular, if we choose $\Sigma_j$ such that
\begin{equation}\label{eq: lambda1 growth}
\lambda_1(\Sigma_j)
\geq
\frac{(n-1)^2a_j^2}{4}e^{2a_jL}.
\end{equation}
Note that when $n\ge 3$, we can choose $\Sigma_j$ with $|\Sigma_j|=1$, ; see e.g. \cite{CD94,CES03}. When $n=2$, we can choose $\Sigma_j$ to be a circle of length $|\Sigma_j|=\frac{2\pi}{a_j}e^{-a_jL}.$ We conclude
\[\mu_1(M_j)L^2
\geq
\frac{(n-1)^2a_j^2}{4}L^2
\to\infty,\qquad a_j\to\infty.
\]
The warping function used above is only Lipschitz at the middle slice
$t=0$, because of the absolute value in the definition of $h_j$. 
This is not essential. One may replace $|t|$ by a smooth even function
 which agrees with $|t|$ outside $(-\varepsilon_j,\varepsilon_j)$ and gives the same estimates and conclusions.
\end{proof}

\begin{remark}\label{rem:Lk-for-M-j} Here, we compute the reduced range $L$ for the above family of manifolds $M_j$ on which Theorem \ref{thm: general result II} holds.
    Near either boundary component, set
$s=L-|t|$, $s\in[0,L]$.
Then $
h_j(t)=e^{a_j(L-|t|)}$
becomes 
$h_j(s)=e^{a_js}.$
Thus, in the inward normal coordinate \(s\), we have
\[
\cR_L=
\int_0^L h_j(s)^{1-n}\,ds
=
\int_0^L e^{-(n-1)a_js}\,ds=
\frac{1-e^{-(n-1)a_jL}}{(n-1)a_j}.
\]
In the notations we introduced in \eqref{eq:warped-Lk}, set $L_1(M_j):=L_1(L)$ and solve $
\cR_{L_1(M_j)}=\frac{\cR_L}{4}$.
Using the formula above, after straightforward calculation, we get
\[
L_1(M_j)
=
\cR^{-1}\left(\frac{\cR_L}{4}\right)
=
\frac{1}{(n-1)a_j}
\log\left(
\frac{4}{3+e^{-(n-1)a_jL}}
\right)\sim \frac{1}{(n-1)a_j}
\log\left(
\frac{4}{3}
\right).
\] 
\end{remark}

\begin{proof}[\normalfont\bf Proof of Theorem \ref{thm:counterexample upper bound}]
 We first  consider the product manifold $M=[-R-1, R+1] \times \Sigma$, $R\gg 1$. And show that  for any connected subdomain $\Omega \in \mathfrak{U}_1$--- a subdomain containing the boundary collar of width $1$--- $\mu_1(\Omega)$  decays faster than $\sigma_1(M)$ as $R\to \infty$.  Note that 
\[
\sigma_1(M) =\frac{1}{R+1}, \qquad \text{when }R\gg1.
\]
To show that $\mu_1(\Omega)$ for any $\Omega\in \mathfrak{U}_1$ has a faster decay, we use the following test function
\[
f(t,.) = \begin{cases} 1 &\text{on $[-R-1, -R]\times \Sigma$},\\ -\frac{t}{R} &\text{on $[-R,R]\times \Sigma$},\\ -1 &\text{on $[R,R+1]\times \Sigma$}.\end{cases}
\] Let $\nu : [-R-1,R+1]\to \mathbb{R}_+$ be the function given by $\nu(t):=|\Omega \cap (\{t\}\times \Sigma)|$ and $\bar f :=\frac{1}{|\Omega|}\int_\Omega f$ be the average of $f$ (here, $|.|$ denotes the volume). From the variational characterisation,
\[
\mu_1(\Omega)\le \frac{\int_\Omega |\nabla (f-\bar f)|^2}{\int_\Omega (f-\bar f)^2}.
\]
The numerator is given by
\[
\int_\Omega |\nabla f|^2 =\frac{1}{R^2}\int_{-R}^R \nu(t)dt.
\]
We set $\eta(t):=\frac{\nu(t)}{\int_{-R}^R\nu(t)\,dt}$ as a probability measure on $[-R,R]$. Hence, for the denominator, we have
\begin{align*}
\int_\Omega (f-\bar f)^2&=\int_\Omega f^2 -|\Omega|\bar  f^2\\
&=2|\Sigma|+\int_{-R}^{R} \left(\frac{t}{R}\right)^2 \,\nu(t)dt - \frac{1}{|\Omega|}\left(\int_{-R}^R \frac{t}{R} \,\nu(t)dt\right)^2\\
&\ge2|\Sigma|+\int_{-R}^{R} \left(\frac{t}{R}\right)^2 \,\nu(t)dt - \frac{1}{\int_{-R}^R\nu(t)}\left(\int_{-R}^R \frac{t}{R} \,\nu(t)dt\right)^2\\
&=2|\Sigma|+\int_{-R}^R\nu(t)\,dt\left(\int_{-R}^{R} \left(\frac{t}{R}\right)^2 \,\eta(t)dt - \left(\int_{-R}^R \frac{t}{R} \,\eta(t)dt\right)^2\right)\\
&=2|\Sigma|+\frac{\var_\eta(t)}{R^2}\int_{-R}^R\nu(t)\,dt.
\end{align*}
To minimise $\var_\eta(t)$, we need to put most of the mass around the mean. Let $$\bar t=\int_{-R}^R t\,\eta(t)\,dt,\qquad m:=\int_{-R}^R\nu(t)\,dt
.$$ Then
\begin{equation*}
\var_\eta(t)=\int_{-R}^R(t-\bar t)^2\eta(t)\,dt\ge  
\frac{m^2}{12|\Sigma|^2}.   
\end{equation*}
One can view the above inequality as a consequence of the Bathtub principle, see e.g.~\cite[Theorem 1.14]{LL01}. Essentially, the argument goes as follows:  since $\var_\eta(t)=\var_\eta(t-\bar t)$, we can assume $\bar t=0$ and view $\eta$ as probability measure on $\mathbb{R}$ by extending it by zero.
\[\var_\eta(t)=\int_{-\infty}^\infty t^2\eta(t) dt =\int_{-\infty}^\infty \int_0^{|t|} 2s\,ds\,\eta(t)\, dt=\int_0^\infty2s\,\int_{|t|\ge s}  \eta(t)\,dt\, ds.\]
Let $\eta(t)\le M$. Then $\int_{|t|<s}\eta(t)\,dt\le 2Ms$. Hence, $\int_{|t|\ge s}\eta(t)\,dt\ge (1- 2Ms)\mathds{1}_{\{1-2Ms\ge0\}}$. 
We then get
\[\var_\eta(t)\ge \frac{1}{12 M^2}.\]
Therefore,
\begin{equation*}
 \int_\Omega ( f-\bar  f)^2\ge 2|\Sigma|+\frac{m^3}{12|\Sigma|^2R^2}.
 \end{equation*}
 Putting the bounds together, we obtain
 \[\mu_1(\Omega)\le \frac{\frac{m}{|\Sigma|}}{2R^2+\frac{m^3}{12|\Sigma|^3}}.\]
 Note that $\frac{m}{|\Sigma|}\in[0,2R]$. The   right-hand side as a function of $s=\frac{m}{|\Sigma|}$ attains its maximum at $s=(12R^2)^{1/3}$ on interval $[0,2R]$, assuming $R\ge 3/2$. Hence, 
\[\mu_1(\Omega)\le\frac{12^{\frac{1}{3}}}{3R^{\frac 43}}.\]
Recall that
\[
\sigma_1(M) =\frac{1}{R+1}, \qquad \text{when }R\gg1,
\]
that is, $\mu_1(\Omega)$ decays faster than $\sigma_1(M)$ as $R\to \infty$.

To complete the proof of the statement of the theorem, we can perturb the product metric $g$ in such a way that the normal injectivity radius of $(\partial M,\tilde g)$, where $\tilde g$ is the perturbed metric, remains bounded between $1$ and $1+\epsilon_0$ for a fixed $\eps_0>0$, with $g$ and $\tilde g$ quasi-isometric with a constant $Q\ge 1$ independent of $R$, which ensures that the decay rates of $\sigma_1(M,\tilde g)$ and $\mu_1(\Omega,\tilde g)$ remain unchanged as they satisfy $\sigma_1(M,g)\le Q^{n+\half}\sigma_1(M,\tilde g)$ and $\mu_1(\Omega,\tilde g)\le Q^{n+1}\mu_1(\Omega,g)$.
\begin{figure}[H]
    \centering
    \includegraphics[width=0.3\linewidth]{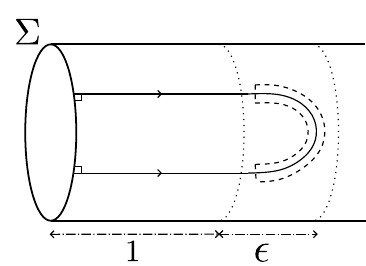}
    \caption{The solid lines orthogonal to the boundary are geodesics with respect to the perturbed metric $\tilde g$. The perturbation happens only inside a small tubular neighbourhood of the joining curve contained within a distance between $1$ and $1+\eps$ along the axis from a boundary component, once for all $R\gg 1$.}
    \label{fig: normal inj rad}
\end{figure}
The normal injectivity radius is highly sensitive to perturbations of the metric, and there are several ways to construct such perturbations.
For instance, we can proceed as follows. Consider two normal geodesics originating from the boundary component $\{-R-1\}\times \Sigma$, close to each other and of length $1$, {(see Figure~\ref{fig: normal inj rad})}. Join them at the other end smoothly by a curve $\Gamma$ contained in the region $[-R,-R+\eps]\times \Sigma$, for a small $\eps>0$ fixed. It would then be possible to locally perturb the product metric $g$ conformally to a metric $\tilde g$ with respect to which $\Gamma$ becomes a geodesic, with $g$ and $\tilde g$ agreeing outside a small tubular neighbourhood of $\Gamma$ contained in $[-R,-R+\eps]\times \Sigma$. Here, the conformal factor (and hence the associated quasi-isometry constant) can be written explicitly in terms of the geodesic curvature of $\Gamma$ with respect to the original metric $g$, independent of $R$. Thus, the normal exponential map with respect to $\tilde g$ ceases to be a diffeomorphism beyond the mid-point of $\Gamma$ with respect to $\tilde g$, ensuring that the normal injectivity radius of $\partial M$ with respect to $\tilde g$ lies within $(1,1+\eps_0)$, for a fixed $\eps_0>0$. We fix this perturbation once for all $R$, and hence, the quasi-isometry constant $Q$ is independent of $R$, thereby leaving the decay rates of $\sigma_1(M,g)$ and $\mu_1(\Omega,\tilde g)$ unaffected as described earlier.
\end{proof}
\section{Proof of the {Main Results---General Setting}}\label{sec: proof of main theorem}
Roughly speaking, to prove the lower bound in Theorems \ref{thm:general result} and \ref{thm: general result II}, we use the eigenfunctions corresponding to $\{\sigma_i(M)\}_{i=1}^k$ as test functions for $\mu_k(\Omega)$. The main point is then to relate the Rayleigh quotient associated with the Steklov problem to the Rayleigh quotient associated with the Neumann problem on a subdomain of the manifold containing a collar neighborhood of the boundary. This is achieved through the trace inequality and a suitable geometric estimate for the trace constant.  The trace inequality 
$$\|f\|_{L^2({\partial M)}}^2\le C(M)\left(\|f\|^2_{L^2(M)}+\|\nabla f\|^2_{L^2(M)}\right),$$
where $C(M)$ is some geometric constant depending on $M$
holds even for manifolds with Lipschitz boundary. Here, we assume that the boundary is smooth, although our proof may still hold under weaker regularity assumptions.

The key point is to obtain a good geometric bound on $C(M)$, or more precisely on the constants $A$ and $B$ appearing in the following inequality, which depend on the geometry of the boundary collar of width $L$.

\begin{equation}\tag{$\star$}\label{eq:star}
\int_{\partial M} f^2
  \le A\int_{\mathcal{C}_L} f^2
      + B\int_{\mathcal{C}_L} |\nabla f|^2 .
\end{equation}
\begin{theorem}\label{thm:lowerbound by abstract trace ineqaulity}
    Let $M$ be a Riemannian manifold with boundary. Let $L\in (0,\injM]$. Assume that \eqref{eq:star} holds for some $A,B>0$. Then
    \[
\sigma_k(M)\ge \frac{1}{2k}\min\left\{\frac{1}{B},\frac{1}{A}\sup_{\Omega\in{\mathfrak{U}_L}}\mu_k(\Omega)\right\},\qquad k\ge1.
\]
\end{theorem}
\begin{proof}
    Let $\psi_1,\dots,\psi_k$ be the associated Steklov eigenfunctions
corresponding to $\sigma_1,\dots,\sigma_k$, and assume they are
$L^2(\partial M)$-orthonormal.
Let
\[
\bar{\psi}_i:=\frac{1}{|\Omega|}\int_\Omega \psi_i.
\]
We consider $
V_k:=\operatorname{span}\{\psi_1-\bar{\psi}_1,\dots,\psi_k-\bar{\psi}_k\}.$
Note that $
\int_\Omega f=0$ for every $f\in V_k$.
Hence
\begin{equation}\label{eq: testfn ineq for mu1 warped}
\mu_k(\Omega)\le \sup_{f\in V_k\setminus\{0\}} \frac{\int_\Omega |\nabla f|^2}{\int_\Omega f^2}.
\end{equation}
Let $
f=\sum_{i=1}^k a_i(\psi_i-\bar{\psi}_i)$.
We now estimate the Rayleigh quotient in   \eqref{eq: testfn ineq for mu1 warped}, For the numerator, we have
\begin{align}\label{eq:nabla f estimate}
\nonumber\int_\Omega |\nabla f|^2
&=
\sum_{i,j=1}^k a_i a_j \int_\Omega \langle \nabla\psi_i,\nabla\psi_j\rangle \\
\nonumber&\le
\sum_{i=1}^k a_i^2 \int_\Omega |\nabla\psi_i|^2
+\sum_{i\ne j}|a_i a_j|
\left(\int_\Omega | \nabla\psi_i|^2\right)^{1/2}\left(\int_\Omega|\nabla\psi_j|^2\right)^{1/2} \\
\nonumber&\le
\sigma_k\sum_{i=1}^k a_i^2
+\sigma_k\sum_{i\ne j}|a_i a_j|\\
\nonumber&\le \sigma_k\left(\sum_i|a_i|\right)^2\\
&\le k\sigma_k\sum_ia_i^2
\end{align}
Now, we obtain a lower bound for the denominator. Let  $u=\sum_ja_j \psi_j$. Note that $f=u-\bar u$, where $\bar u=\sum_ja_j \bar\psi_j$.  Since  $\int_{\partial M} u=0$,
\[\sum_ja_j^2=
 \int_{\partial M} u^2\le \int_{\partial M} (u-\bar u)^2= \int_{\partial M} (u^2+\bar u^2)=\int_{\partial M} f^2.
\]
Using the trace inequality~\eqref{eq:star} and the fact that $\Omega$ contains $\mathcal{C}_L$, we get
\begin{align}\label{eq: bound for int f2}
\nonumber\sum_ia_i^2=\int_{\partial M} u^2 \le \int_{\partial M} (u-\bar u)^2 &\le A\int_{\Omega} (u-\bar u)^2 +  B\int_\Omega |\nabla (u-\bar u)|^2\\
\nonumber&\le A\int_{\Omega} f^2 +B \int_M|\nabla f|^2 \\
&\le A\int_{\Omega} f^2 +  B\,\,k\sigma_k(M)\sum_ia_i^2,
\end{align}
where the last inequality follows from  inequality \eqref{eq:nabla f estimate}.\\
If $B\,k\,\sigma_k(M)\le \half$, we have
\[
\sum_j a_j^2 \le 2A\int_\Omega f^2,
\]
whence it follows from~\eqref{eq: testfn ineq for mu1 warped} and \eqref{eq:nabla f estimate} that
\[
\sigma_k(M) \ge \frac{\mu_k(\Omega)}{2A\,k}.
\]
Otherwise, we have
\[
\sigma_k(M) \ge \frac{1}{2B \,k}.
\]
\end{proof}
The following lemmas gives desired estimate for $A$ and $B$ that can be used to proof Theorems \ref{thm:general result} and  \ref{thm: general result II}.
\begin{lemma}\label{lem: trace estimate for general collars}
{Let $L\in (0,\injM]$.} Then  for any $f\in C^\infty(\mathcal{C}_L)$, and  the following inequality holds.
\[
\int_{\partial M} f^2
  \le 2\cR_L\left(\frac{1}{L^2}\int_{\mathcal{C}_L} f^2
      +\int_{\mathcal{C}_L} |\nabla f|^2\right).
\]
\end{lemma}
\begin{proof}
In Fermi coordinates, we have on each $\mathscr{C}_j^L$,
\[
f(0,x)=f(t,x)-\int_0^t \pa_s f\,ds,\qquad \forall\,t\in [0,L].
\]
By the Cauchy--Schwarz inequality, we obtain
\begin{align*}
f(0,x)^2
  &\le 2f(t,x)^2+2\left(\int_0^t \pa_s f\,ds\right)^2\\
  &\le 2f(t,x)^2
      +2\int_0^L \varphi_j^{-1}\,ds
       \cdot\int_0^L \varphi_j\,|\pa_s f|^2\,ds.
\end{align*}
Averaging  both sides with respect to $\frac{\varphi_j\,dt}{\int_0^L \varphi_j\,ds}$, and setting $A_{j}(x):=\int_0^L \varphi_j^{-1}\,ds $  we get
\begin{align}\label{eq:AL bound}
   f(0,x)^2&\le \frac{2}{\int_0^L\varphi_j\,ds}\int_0^Lf^2\varphi_j\,ds
      +2A_j(x)\int_0^L \varphi_j\,|\pa_s f|^2\,ds\\
\nonumber      &\le  \frac{2A_j(x)}{L^2}\int_0^Lf^2 f_j\,ds
      +2A_j(x)\int_0^L \varphi_j\,|\nabla f|^2\,ds,
\end{align}
where in the last inequality we used the Cauchy--Schwarz inequality,
$\int_0^L \varphi_j \cdot\int_0^L \varphi_j^{-1} \ge L^2$. 
Using  $A_j(x)\le \cR_L$ and integrating both sides over
$\Sigma_j$,  we conclude
\begin{align*}
\int_{\Sigma_j} f^2
  &\le 2\cR_L\left( \frac{1}{L^2}\int_{\mathscr{C}_j^L} f^2
      + \int_{\mathscr{C}_j^L}|\nabla f|^2\right).
\end{align*}
Summing over $j\in\{1,\dots,b\}$,
we obtain the desired inequality.
\end{proof}
\begin{remark}
Rather than using the upper bound $A_j(x)/L^2$ for $\int_0^L \varphi_j^2$ in \eqref{eq:AL bound}, one could instead use a better bound
\[
\max_j \sup_{x\in \Sigma_j} \left(\int_0^L \varphi_j\right)^{-1}.
\]
We use the former bound only for simplicity of presentation.
\end{remark}
\begin{proof}[\normalfont\textbf{Proof of  Theorem \ref{thm:general result} }]
    It immediately follows from Theorem \ref{thm:lowerbound by abstract trace ineqaulity} and Lemma \ref{lem: trace estimate for general collars}.
\end{proof}
To prove Theorem \ref{thm: general result II}, we still use Lemma
\ref{lem: trace estimate for general collars}, but we need to modify the last
step in the proof of Theorem \ref{thm:lowerbound by abstract trace ineqaulity}.

{\begin{proof}[\normalfont\textbf{Proof of Theorem \ref{thm: general result II}}] 
For any $L\in (0,\inj^{\perp}{\partial M})$, by the proof of Theorem \ref{thm:lowerbound by abstract trace ineqaulity} and Lemma \ref{lem: trace estimate for general collars}, we have inequalities \eqref{eq:nabla f estimate} and \eqref{eq: bound for int f2}:
\[\int_\Omega |\nabla f|^2\le k\sigma_k\int_\Sigma f^2,\qquad \int_\Sigma f^2 \le2{\cR}_L\left(\frac{1}{L^2}\int_{\Omega} f^2 +k\,\sigma_k(M)\int_\Sigma f^2\right).
\]
The aim is to find a range of $L$ for which $2\cR_L\,k\,\sigma_k(M)\le \frac12$. For this range of $L$, the second inequality gives the upper bound $\int_{\Sigma} f^2 \le \frac{4\cR_L}{L^2}\int_\Omega f^2$. Substituting this bound into the first inequality gives the desired conclusion.

We first give an upper  bound for $\sigma_k(M)$. We choose $k+1$ distinct boundary components
$\Sigma_{i_0},\ldots,\Sigma_{i_k}$. For each $j=0,\ldots,k$, let
$\cC_{i_j}^{L_0}$ be the collar of $\Sigma_{i_j}$ of width $L_0=\injM$. In Fermi coordinates
$(t,x)\in[0,L_0]\times\Sigma_{i_j}$, recall $
dv_g=\varphi_{i_j}(t,x)\,dt\,dv_{\Sigma_{i_j}}(x),$ and $
\overline\varphi_{i_j}(t)
=
\frac{1}{|\Sigma_{i_j}|}
\int_{\Sigma_{i_j}}\varphi_{i_j}(t,x)\,dv_{\Sigma_{i_j}}(x).
$
Set
\[
\overline \cR_{i_j}:=\int_0^{L_0} \overline\varphi_{i_j}(t)^{-1}\,dt .
\]
On $\cC_{i_j}^{L_0}$, we take a radial function $u_{i_j}=u_{i_j}(t)$ such that
$u_{i_j}=1$ on $\Sigma_{i_j}$ and $u_{i_j}=0$ at $t=L$, and then extend it
by zero outside $\cC_{i_j}^{L_0}$. For a radial function $u=u(t)$ with this boundary condition we have 
\[
\int_{\cC_{i_j}^{L_0}}|\nabla u|^2\,dv_g
=
|\Sigma_{i_j}|\int_0^{L_0} |u'(t)|^2\overline\varphi_{i_j}(t)\,dt,
\]
and by Cauchy--Schwarz,
\[
1=u(0)-u(L)=-\int_0^L u'(t)\,dt
\leq
\left(\int_0^L |u'(t)|^2\overline\varphi_{i_j}(t)\,dt\right)^{1/2}
\left(\int_0^L \overline\varphi_{i_j}(t)^{-1}\,dt\right)^{1/2}.
\]
Therefore
\[
\int_0^L |u'(t)|^2\overline\varphi_{i_j}(t)\,dt
\geq
\frac{1}{\overline \cR_{i_j}}.
\]
Equality holds for
\[
u_{i_j}(t)
=
1-
\frac{1}{\overline \cR_{i_j}}
\int_0^t \overline\varphi_{i_j}(s)^{-1}\,ds,
\]
which is equivalently the solution of
$\big(\overline\varphi_{i_j}(t)u_{i_j}'(t)\big)'=0$ with
$u_{i_j}(0)=1$ and $u_{i_j}(L)=0$. Hence
\[
\int_M|\nabla u_{i_j}|^2\,dv_g
=
\frac{|\Sigma_{i_j}|}{\overline \cR_{i_j}},
\qquad
\int_{\partial M}u_{i_j}^2\,dv_{\partial M}
=
|\Sigma_{i_j}|.
\]
We apply the min--max principle to the mutually disjointly supported functions
\(\{u_{i_0},\ldots,u_{i_k}\}\) to obtain
\begin{equation}\label{eq:bound-for-sigmak}
\sigma_k(M)
\leq
\max_{0\leq j\leq k}\frac{1}{\overline \cR_{i_j}}\le \frac{1}{\underline \cR_{L_0}}:=\frac{1}{\min_{1\le j\le b} \overline{\cR}_{j}}.
\end{equation}
By inequality \eqref{eq:overline AL and BL bounds AL} and \eqref{eq:bound-for-sigmak}
\[2\cR_L\,k\,\sigma_k(M)\le 2\tau\, \overline{\cR}_L\,k\,\sigma_k\le \frac{2\tau\overline \cR_Lk}{\underline \cR_{L_0}}.\]
Thus, to have $2\cR_L\,k\,\sigma_k(M)\le \half$, it is enough to take  $L$ be such that $$\cR_L={\overline \cR_L}\le \frac{\underline \cR_{L_0}}{4\,\tau\,k}.$$
Note that $\cR$ is increasing and for any  $L\le \cR^{-1}\left(\frac{\underline \cR_{L_0}}{4\tau k}\right)$ we get desired lower bound.
\end{proof}}
We now give a different geometric estimate for $A$ and $B$ in inequality \eqref{eq:star}. The following lemma is inspired by an identity used in \cite{KS68,HS20} for star-shaped domains. Roughly speaking, we use the same identity, but replace the distance from the centre with the distance from the boundary. \\
\begin{lemma}\label{lem:trace type II} Under the assumption of Theorem \ref{thm: main result second version}, for every $f\in C^\infty(M)$ with $\int_\Omega f=0$, 
  \[
\int_{\partial M} f^2
  \le \frac{1+nL{\kappa_{\max}}}{L}\int_{\mathcal{C}} f^2
      + \frac{2}{\sqrt{\mu_1(\Omega)}}\int_{\mathcal{C}} |\nabla f|^2 .
\]  
\end{lemma}

\noindent Theorem \ref{thm: main result second version} then immediately follows from Lemma \ref{lem:trace type II} and Theorem  \ref{thm:lowerbound by abstract trace ineqaulity}. 

\begin{proof}[Proof of Lemma \ref{lem:trace type II}] 
Define $
\rho(x) = \frac{1}{2}\bigl(L - \dist(x,\Sigma)\bigr)^2.$
Then
\[
\nabla\rho = -\bigl(L - \dist(x,\Sigma)\bigr)\,\nabla\!\dist, 
\qquad 
\nabla\rho\big|_{\Sigma} = L\,\nu,
\]
and for any $f \in C^\infty(M)$, integration by parts gives
\[
\frac{L}{2}\int_{\partial M} f^2 
=\frac{1}{2}\int_{\partial M} f^2\langle \nu, \nabla\rho \rangle= \int_{\mathcal{C}} f\,\langle \nabla f, \nabla\rho \rangle  
- \frac{1}{2}\int_{\mathcal{C}} f^2 \, \Delta\rho.
\]
The above identity is also used in \cite[Theorem 1.3]{HS20}.
Since $|\nabla\rho| = L - \dist \leq L$, by the Cauchy--Schwarz inequality, we get
\[
\int_{\mathcal{C}} f\,\langle \nabla f, \nabla\rho\rangle
\leq L\int_{\mathcal{C}} |f|\,|\nabla f| 
\leq L\left(\int_{\mathcal{C}} f^2\right)^{\!1/2}
\left(\int_{\mathcal{C}} |\nabla f|^2\right)^{\!1/2}.
\]
Since $\int_\Omega f=0$, we have  $\int_\Omega f^2\, \leq \mu_1^{-1}\int_\Omega |\nabla f|^2$. Therefore,
\[
\int_\mathcal{C} f\,\langle \nabla f, \nabla\rho\rangle
\leq L\,\mu_1^{-1/2}\int_M |\nabla f|^2.
\]
At a point $y$ at distance $t$ from $\Sigma$, the eigenvalues of $\nabla^2\rho$ are
\[
(L-t)\,\kappa_1(t),\;\ldots,\;(L-t)\,\kappa_n(t),\;1,
\]
where $\kappa_1(t) \leq \cdots \leq \kappa_n(t)$ are the principal curvatures of $\Sigma_t$. Hence
\[
{-}\Delta\rho = 1 + (L-t)\sum_{i=1}^{n}\kappa_i(t).
\]
If $\kappa_i(t) \leq \kappa_{\max}$ for all $i$ and $t \in [0,L)$, then
\[
{-}\Delta\rho \leq 1 + nL\,\kappa_{\max}.
\]
We conclude
\[\frac{L}{2}\int_{\partial M} f^2
\le \frac{L}{\sqrt{\mu_1}}\int_M |\nabla f|^2+\frac{1 + nL\,\kappa_{\max}}{2}\int_M f^2
\]
\end{proof}

\appendix \section{Curvature Bounds for the Volume Density Ratio {and Resistance}}\label{sec: volume density estimate}
The main geometric quantities appearing in our main results are defined in terms of the volume density ratio and resistance. The purpose of this appendix is to use the Jacobi field comparison theorem to derive estimates for these quantities in terms of the normal injectivity radius, sectional-curvature bounds, and bounds on the principal curvatures of the boundary. Our main objective is to prove Proposition~\ref{prop: bounds-in-terms-of-curvature}.
\begin{proof}[{Proof of Proposition~\ref{prop: bounds-in-terms-of-curvature}}]
In the notation of Theorem~\ref{thm: general result II}, choose $L_0 = \half \min \{L^\star, \injM\}$, where $L^\star(n,K_{\max},\kappa_{\max},\injM)$ is specified later in \eqref{eq:Lstar-general}. Using Theorem~\ref{thm: general result II} on the collar $\mathcal{C}_{L_0}$, we get $L_{b-1}(\tau_{L_0},L_0)\in (0,L_0]$ such that for any $L\in (0,L_{b-1}]$,
\begin{equation} \label{eq:app lbd}
\sigma_k(M) \ge \frac{1}{4\cR_Lk}\sup_{\Omega\in{\mathfrak{U}_L}}\mu_k(\Omega)\, L^2, \qquad 1\le k\le b-1.
\end{equation}
Here, we use that $L_k$ is decreasing in $k$, by definition. We then use the volume density estimates of ~\eqref{eq:phi-2sided} below to have an upper bound on $\tau_{L_0}(K_{\min}, K_{\max},\kappa_{\min},\kappa_{\max},\injM)$, whence we define $\hat L$ depending on the these quantities and $b$. Finally, we provide $C(n,K_{\max}, \kappa_{\max},\injM)$ that acts as a lower bound for $\frac{1}{\cR_L}$ giving that, for $L\in (0,\hat L]$,
\[
\sigma_k(M) \ge \frac{{C}}{4k}\sup_{\Omega\in{\mathfrak{U}_L}}\mu_k(\Omega)L^2, \qquad 1\le k\le b-1.
\]

We use comparison theorems for the volume density ratio $\varphi_j(x,t)$. 
Such results are standard; see, for instance \cite{HK78,Kar89,EH90}, and Gray's book \cite{G03}.
Pointwise estimates for $\varphi_j$ yield bounds on $\cR_L$ in terms of curvature bounds.
We have already discussed a special case of this estimate for negatively curved manifolds.

Fix a boundary component $\Sigma=\Sigma_{j}$ and $x\in\Sigma$. Along the
inward normal geodesic $\gamma_{x}(t)=\exp_{x}(-t\nu)$ parallel–transport
an orthonormal frame $\{E_{1}(t),\dots,E_{n-1}(t)\}$ of
$T_{\gamma_{x}(t)}\Sigma_{j}^{t}$ with $E_{i}(0)=e_{i}\in T_{x}\Sigma$, and
let $A(t)\in\mathrm{End}(T_{x}\Sigma)$ be the matrix solving the
boundary–Jacobi equation
\begin{equation}
A''(t)+\mathcal R(t)\,A(t)=0,
\qquad
A(0)=I,\qquad
A'(0)=-\,\mathrm{II}_{\Sigma,\nu}(x),
\end{equation}
where $\mathcal R(t)_{ik}=\bigl\langle R\bigl(E_{i}(t),\dot\gamma_{x}(t)\bigr)\dot\gamma_{x}(t),E_{k}(t)\bigr\rangle$
and $\mathrm{II}_{\Sigma,\nu}$ is the second fundamental form of $\Sigma$
with respect to the outward unit normal $\nu$. Then
\[
\varphi_{j}(t,x)\;=\;\bigl|\det A(t)\bigr|.
\]
Equivalently, with $U(t)=A'(t)A(t)^{-1}$ (the Weingarten map of
$\Sigma_{j}^{t}$ in the direction $\partial_{t}$),
\begin{equation}
U'(t)+U(t)^{2}+\mathcal R(t)=0,\qquad
U(0)=-\,\mathrm{II}_{\Sigma,\nu}(x),\qquad
\partial_{t}\log\varphi_{j}(t,x)=\mathrm{tr}\,U(t).
\end{equation}
For $K\in\mathbb R$, define
\[
s_K(r)=
\begin{cases}
\frac{\sin(\sqrt K\,r)}{\sqrt K}, & K>0,\\
r, & K=0,\\
\frac{\sinh(\sqrt{-K}\,r)}{\sqrt{-K}}, & K<0,
\end{cases}
\qquad \text{and}\qquad
c_K(r)=s_K'(r).
\]
These solve $f''+K f=0$ with $(f(0),f'(0))=(0,1)$ and $(1,0)$ respectively. Consider the model Jacobi factors
\[
S_-(t):=
c_{K_{\max}}(t)-\kappa_{\max}s_{K_{\max}}(t),
\qquad\text{and}\qquad
S_+(t):=
c_{K_{\min}}(t)-\kappa_{\min}s_{K_{\min}}(t).
\]
Define the first focal radii
\[
L^\star
:= \sup\{t>0:\ S_-(s)>0
\text{ for all }s\in[0,t]\}.
\]
Recall from the statement of Proposition~\ref{prop: bounds-in-terms-of-curvature} that
on $\mathcal{C}_{\injM}$, 
\begin{itemize}
    \item[(H1)] {$K_{\min}\le \mathrm{Sec}_{g}\le K_{\max}$ for some
$K_{\min},K_{\max}\in \RR$, and}
\item[(H2)] the principal curvatures of $\Sigma_{j}$ with respect to the
outward normal $\nu$ satisfy\\
$
\kappa_{\min}\le \kappa_{i}(x)\le \kappa_{\max},\;\;1\le i\le n-1,\;x\in\Sigma_{j},
$
for some $\kappa_{\min},\kappa_{\max}\in \mathbb{R}$.
\end{itemize}
Solving $S_-(t)=0$, we obtain
\begin{equation}\label{eq:Lstar-general}
L^\star=
\begin{cases}
\frac{1}{\sqrt {K_{\max}}}
\arctan\!\left(\frac{\sqrt {K_{\max}}}{\kappa_{\max}}\right),
& {K_{\max}}>0,\ \kappa_{\max}>0,\\
\frac{\pi}{2\sqrt {K_{\max}}},
& {K_{\max}}>0,\ \kappa_{\max}=0,\\
\frac{1}{\sqrt {K_{\max}}}
\left(
\pi-\arctan\!\left(\frac{\sqrt {K_{\max}}}{|\kappa_{\max}|}\right)
\right),
& {K_{\max}}>0,\ \kappa_{\max}<0,\\
\frac{1}{\kappa_{\max}},
& {K_{\max}}=0,\ \kappa_{\max}>0,\\
+\infty,
& {K_{\max}}=0,\ \kappa_{\max}\le0,\\
\frac{1}{\sqrt{|{K_{\max}}|}}
\operatorname{arctanh}\!\left(\frac{\sqrt{|{K_{\max}}|}}{\kappa_{\max}}\right),
& {K_{\max}}<0,\ \kappa_{\max}>\sqrt{|{K_{\max}}|},\\
+\infty,
& {K_{\max}}<0,\ \kappa_{\max}\le\sqrt{|{K_{\max}}|}.
\end{cases}
\end{equation}
We henceforth assume:
\begin{itemize}
\item[(H3)] $L<\min\{L^\star,\injM\}$, where $\injM$ is the normal injectivity radius of $\partial M$.
\end{itemize}

\begin{proposition}\label{prop:riccati-comp}
Under the hypotheses \textup{(H1)--(H3)}, each principal curvature
$\kappa_i(t,x)$ of the level set $\Sigma_j^t$ satisfies
\begin{equation*}
  u_-(t) \;\le\;\kappa_i(t,x)\;\le\;u_+(t)
\end{equation*}
for every $(t,x)\in[0,L]\times\Sigma_j$,
where $u_-$ and $u_+$ satisfy the model Riccati equations
\[
u_-' + u_-^2 + K_{\max} = 0,
\qquad
u_-(0)=-\kappa_{\max}.
\]
and
\[
u_+' + u_+^2 + K_{\min} = 0,
\qquad
u_+(0)=-\kappa_{\min}.
\]
Consequently, 
\begin{equation}\label{eq:phi-2sided}
S_-(t)^{n-1}
\le
\varphi_j(t,x)
\le
S_+(t)^{n-1}.
\end{equation}
\end{proposition}
\noindent We also refer the reader to \cite{CGH20}, where Riccati comparison theorems are used.

\medskip \noindent
Thus, we get under assumptions \textup{(H1)--(H3)} that,
\begin{equation}
\cR_{L}\;\le\;\int_{0}^{L}\!\!S_{-}(t)^{-(n-1)}\,dt
\;=:\;\mathscr{R}(K_{\max},\kappa_{\max},L,n).
\end{equation}
Since $L_0<L^\star$, the estimate~\eqref{eq:phi-2sided} holds on $[0,L_0]$. Averaging in $x$, we get
\[
S_-(t)^{n-1}\le \overline\varphi_j(t)\le S_+(t)^{n-1},\qquad 0\le t\le L_0.
\]
Hence,
\[
\tau_{L_0}(M)\le \max_{0\le t\le L_0}\left(\frac{S_+(t)}{S_-(t)}\right)^{n-1}.
\]
Therefore, $\tau_{L_0}(M)$ is bounded in terms of $n,K_{\min},K_{\max},\kappa_{\min},\kappa_{\max}$ and $\injM$. Recall that $L_{b-1}(\tau_{L_0},L_0)$ is given by
\[
\overline{\cR}_{L_{b-1}}
=
\frac{\underline{\cR}_{L_0}}{4(b-1)\tau_{L_0}}.
\]
Now define
\[
\hat L:=
\sup\left\{
L\in(0,L_0]:
\int_0^L S_-(t)^{-(n-1)}\,dt
\le
\frac{1}{4(b-1)\tau_{L_0}}
\int_0^{L_0}S_+(t)^{-(n-1)}\,dt
\right\},
\]
which depends only on $n$, $K_{\min}$, $K_{\max}$, $\kappa_{\min}$, $\kappa_{\max}$, $\injM$ and $b$. Then the lower bound $\eqref{eq:app lbd}$ holds for all $L\in(0,\hat L]$, since \hbox{$\hat L \in (0,L_{b-1}]$}. Indeed, if $L$ belongs to the set on the right-hand side in the above definition of $\hat L$, then
\[
\overline{\cR}_L
\le
\int_0^L S_-(t)^{-(n-1)}\,dt
\le
\frac{1}{4(b-1)\tau_{L_0}}
\int_0^{L_0}S_+(t)^{-(n-1)}\,dt
\le
\frac{\underline{\cR}_{L_0}}{4(b-1)\tau_{L_0}}
=
\overline{\cR}_{L_{b-1}},
\]
and since $L\mapsto\overline{\cR}_L$ is increasing, this implies $L\le L_{b-1}$. Finally, we may define
\[
C:=\left(\int_0^{L_0} S_-(t)^{-(n-1)}\,dt\right)^{-1},
\]
which serves as a uniform lower bound for $\frac{1}{\cR_L}$ for all $L\in (0,\hat L]$ and depends only on $n$, $K_{\max}$, $\kappa_{\max}$ and $\injM$.
\end{proof}

\begin{remark}
  Similarly, by Proposition \ref{prop:riccati-comp}, one may replace $\cR_L$ with its upper bound $\mathscr{R}(K_{\max},\kappa_{\max},L,n)$ and replace $\injM$ with $\min\{\injM,L^\star\}$ in Theorem \ref{thm:general result}. It also provides an  estimate for $L_0$ in terms of curvature bounds and $\injM$ in Proposition \ref{cor:k-th stek lower bd for warped bdry}.
\end{remark}

\section*{Acknowledgement}
The authors would like to thank the Isaac Newton Institute for Mathematical Sciences (INI), Cambridge, for support and hospitality during the programme Geometric spectral theory and applications, where work on this paper was undertaken. This work was supported by EPSRC grant EP/Z000580/1. T.C. acknowledges the UK spectral theory network mini-grant---supported by the INI and the EPSRC grant EP/V521929/1---that facilitated a visit to the University of  Neuch\^{a}tel.  A.H. was partially supported by a grant from the Simons Foundation.

\end{document}